\documentclass[12pt, showkeys]{article}
\usepackage{amsmath, amssymb, ascmac}
\usepackage{bm}
\usepackage[abbrev]{amsrefs}
\usepackage{enumerate}
\usepackage{amsthm}
\usepackage{comment}
\usepackage{color}
\usepackage[dvipdfmx]{graphicx}
\usepackage[subrefformat=parens]{subcaption}

\usepackage{hyperref}

\usepackage[top=30truemm, bottom=30truemm, left=20truemm, right=20truemm]{geometry}

\newtheorem{theorem}{\bf Theorem}
\newtheorem{lem}{\bf Lemma}[section]
\newtheorem{cor}{\bf Corollary}
\newtheorem{remark}{\rm REMARK}[section]

\newtheorem{prop}{\bf Proposition}[section]

\title{{\huge On the weak convergence of conditioned Bessel bridges}}
\author{{\large Kensuke Ishitani, Tokufuku Rin and Shun Yanashima}}
\date{}
\begin{document}
\maketitle

\begin{abstract}
The purpose of this paper is to introduce the construction of a 
stochastic process called ``$\delta$-dimensional Bessel house-moving" and its properties. 
We study the weak convergence of 
$\delta$-dimensional Bessel bridges conditioned from above, 
and we refer to this limit as $\delta$-dimensional Bessel house-moving. 
Applying this weak convergence result, 
we give the decomposition formula for its distribution 
and the Radon-Nikodym density for the distribution of the Bessel house-moving with respect to the one of the Bessel process. 
We also prove that $\delta$-dimensional Bessel house-moving is a 
$\delta$-dimensional Bessel process hitting a fixed point for the first time at $t=1$. 

\bigskip
\noindent {\bf Keywords:}
barrier options, Greeks, Bessel bridge, Bessel house-moving
\footnote[0]{2010 Mathematics Subject Classification: 
Primary 60F17; Secondary 60J25.}
\end{abstract}


\section{Introduction and main results}\label{section_intro}

The purpose of this paper is to introduce the construction of a 
stochastic process called ``$\delta$-dimensional Bessel house-moving" and its properties. 

Recently, \cite{bib_Ishitani} developed 
a chain rule for Wiener path integrals between two curves 
that arise in the computation of first-order Greeks of barrier options, 
demonstrating its effectiveness with some numerical examples. 
In this chain rule, a $3$-dimensional Bessel bridge and a Brownian meander 
played an important role. 
We believe higher-order chain rules 
might be useful in computing higher-order Greeks of barrier options, 
in which the stochastic process 
``Brownian house-moving'' is expected to play an important role.

The Brownian house-moving is a Brownian bridge 
that stays between its starting point and its terminal point. 
In \cite{bib_3DimMoving},  
it has been proven that the Brownian house-moving can also be obtained 
by the weak convergence of $3$-dimensional Bessel bridges conditioned from above. 
In \cite{bib_Ishitani_2}, a Monte Carlo sampling technique for Brownian house-moving is studied. 
Although the existence of the Brownian house-moving was well-known in \cite{bib_PitmanYor1996, bib_Yano}, 
the weak convergence result for conditioned $3$-dimensional Bessel bridges had yet to be researched, 
and this result is necessary for computing higher-order Greeks of barrier options under the Black-Scholes market model.

Furthermore, to compute higher-order Greeks of barrier options under the general market model, 
we need more general results for the weak convergence of conditioned diffusion bridges. 
As a preparatory step toward this goal, in this paper we focus on the weak convergence of 
conditioned $\delta$-dimensional Bessel bridges. 
We study the weak convergence of 
$\delta$-dimensional Bessel bridges conditioned from above for all $\delta>0$ (Theorem~\ref{Thm_MainResult2}), 
and we refer to this limit as ``$\delta$-dimensional Bessel house-moving". 

Since the Brownian house-moving corresponds to the $3$-dimensional Bessel house-moving, 
our results expand on those in \cite{bib_3DimMoving}. 
In \cite{bib_3DimMoving}, to prove their result, they use a relation 
between the one-dimensional Brownian bridge and the $3$-dimensional Bessel bridge. 
However, we are not able to apply the same approach as \cite{bib_3DimMoving} for $\delta>0$ 
because a $\delta$-dimensional Bessel process 
is abstractly defined as a solution of a stochastic differential equation. 
For this reason, we prove our result with different approaches, such as by 
using estimations related to the Fourier-Bessel expansion (\cite{bib_FourierBessel}, \cite{bib_Serafin}). 

We also prove that a $\delta$-dimensional Bessel house-moving 
is a $\delta$-dimensional Bessel process hitting a fixed point for the first time at $t=1$ (Proposition~\ref{Prop_MainResult1}). 
As mentioned above, 
the first hitting process for one-dimensional diffusion processes already appeared in \cite{bib_PitmanYor1996, bib_Yano}. 
However, since we also construct the Bessel house-moving as the weak limit of conditioned Bessel bridges, 
we can obtain new results on the sample path properties of the Bessel house-moving. 
For example, applying our weak convergence result, 
we can obtain the decomposition formula for its distribution (Theorem~\ref{Thm_Decomp_flat_Moving}) 
and the Radon-Nikodym density for the distribution of the Bessel house-moving 
with respect to one of the Bessel processes (Theorem~\ref{Thm_abs_conti}). 


\subsection{Notations}
We start by introducing notations needed for stating our results. 

Throughout this paper, we fix $\delta>0$ and set $\nu:=\delta/2-1$.

For $0\leq a<b$, let $R^a=\{R^a(t)\}_{t\geq 0}$ be a $\delta$-dimensional Bessel process (BES($\delta$) process for short) starting from $a$. 
In addition, for $0\leq t_1<t_2\leq 1$, 
$r_{[t_1, t_2]}^{a\to b}=\{r_{[t_1, t_2]}^{a\to b}(t)\}_{t\in [t_1, t_2]}$ 
denotes a $\delta$-dimensional Bessel bridge (BES($\delta$) bridge for short) from $a$ to $b$ on $[t_1, t_2]$. 
We write simply 
$r^{a\to b}:=r_{[0, 1]}^{a\to b}$. 

For a continuous process $X$ on $[t_1, t_2]$, 
we denote its maximal value as 
\begin{align*}
&M_{[t_1,t_2]}(X) = \max_{t_1 \leq u \leq t_2}X(u).
\end{align*}
In the case that $[t_1, t_2]=[0,t]$, we write $M_t(X):=M_{[0,t]}(X)$. 
Moreover, in the case that $[t_1, t_2]=[0,1]$, we write simply $M(X):=M_{[0,1]}(X)$.

For $\eta> 0$, $0\leq s<t\leq 1$, and $x, y\in [0, \eta]$, we define
\begin{align*}
q^{(\eta)}_1(s, x, t, y)
:=&\frac{P\left(R^x(t-s)\in dy\right)}{dy}
P\left(M_{[s,t]}(r_{[s,t]}^{x\to y})\leq \eta\right),\\
q^{(\eta)}_2(t, y)
:=&\lim_{\varepsilon \downarrow 0}
\frac{\partial}{\partial \varepsilon}q^{(\eta+\varepsilon)}_1(0, y, t, \eta)
=\frac{P\left(R^y(t)\in d\eta\right)}{d\eta}
\lim_{\varepsilon \downarrow 0}
\frac{\partial}{\partial \varepsilon}P\left(M_{[0, t]}(r_{[0, t]}^{y\to \eta})\leq \eta+\varepsilon \right).
\end{align*}

Let $C([t_1,t_2], \mathbb{R})$ 
be a class of $\mathbb{R}$-valued continuous functions 
defined on $[t_1,t_2]$ and set 
\[
K_{[t_1, t_2]}^-(c)
:=\{w\in C([t_1, t_2], \mathbb{R}) ~|~w(t)\leq c, \ t_1\leq t\leq t_2\}\]
for $c\in\mathbb{R}$. 
In the case that $[t_1, t_2]=[0,1]$, we write simply $K^-(c):=K_{[0,1]}^-(c)$.
Let 
\[
d_{\infty}(w,w') = \| w - w' \|_{C([t_1,t_2], \mathbb{R})} := \sup_{t_1 \leq t \leq t_2} \left| w(t) - w'(t) \right|
\quad (w, w' \in C([t_1,t_2], \mathbb{R})).
\] 
$\mathcal{B}(C([t_1,t_2], \mathbb{R}))$ denotes the Borel $\sigma$-algebra 
with respect to the topology generated by the metric $d_{\infty}$. 
In addition, for $0 \leq s < t \leq 1$, 
$\pi_{[s, t]} : C([0,1], \mathbb{R}) \to C([s,t], \mathbb{R})$ 
denotes the restriction map.

Assume that 
$Y : (\Omega, \mathcal{F}, P) \to 
(C([0,1], \mathbb{R}), \mathcal{B}(C([0,1], \mathbb{R})))$ 
is a random variable and that
$\Lambda \in \mathcal{B}(C([0,1], \mathbb{R}))$ 
satisfies $P(Y \in \Lambda) > 0$. 
Then, we define the probability measure 
$P_{Y^{-1}(\Lambda)}$ on 
$(Y^{-1}(\Lambda), Y^{-1}(\Lambda) \cap \mathcal{F})$ as
\begin{align*}
P_{Y^{-1}(\Lambda)}(A):= \frac{P(A)}{P(Y \in \Lambda)}, 
\qquad A \in Y^{-1}(\Lambda) \cap \mathcal{F}
:= \left\{ Y^{-1}(\Lambda) \cap F~|~F \in \mathcal{F} \right\}.
\end{align*}
Let $Y|_{\Lambda}$ denote the restriction $Y$ to 
$(Y^{-1}(\Lambda), Y^{-1}(\Lambda) \cap \mathcal{F}, P_{Y^{-1}(\Lambda)})$. 
Then, 
\begin{align}\label{Def_conditional_StochProc}
Y|_{\Lambda} : (Y^{-1}(\Lambda), Y^{-1}(\Lambda) \cap \mathcal{F}, P_{Y^{-1}(\Lambda)}) \to 
(\Lambda, \mathcal{B}(\Lambda))
\end{align}
is a random variable. 
Throughout this paper, 
$P_{Y^{-1}(\Lambda)}(Y|_{\Lambda} \in \Gamma)$ is often written as 
$P(Y|_{\Lambda} \in \Gamma)$.

In addition, $X_n \overset{\mathcal{D}}{\to} X$ means that $\{ X_n \}_{n=1}^{\infty}$ converges to $X$ in distribution.

\subsection{Main results}
First, we construct a stochastic process called ``$\delta$-dimensional Bessel house-moving'' (BES($\delta$) house-moving for short) $H^{a \to b}$ as the weak limit of BES($\delta$) bridges conditioned from above. 
\begin{theorem}\label{Thm_MainResult2}
Let $0\leq a<b$. 
There exists an $\mathbb{R}$-valued continuous Markov process 
$H^{a \to b}=\{H^{a \to b}(t) \}_{t \in [0,1]}$ that satisfies
\begin{align*}
r^{a \to b}|_{K^-(b+\eta)} \xrightarrow{\mathcal{D}} H^{a \to b}, 
\qquad \eta \downarrow 0.
\end{align*}
Moreover, for $0<s<t<1$ and $x, y\in (0, b)$, the law of $H^{a\to b}$ is given by
\begin{align*}
P\left(H^{a\to b}(t)\in dy\right)
&=\frac{q^{(b)}_1(0, a, t, y)q^{(b)}_2(1-t, y)}{q^{(b)}_2(1, a)}dy,\\
P\left(H^{a\to b}(t)\in dy~|~H^{a\to b}(s)=x\right)
&=\frac{q^{(b)}_1(s, x, t, y)q^{(b)}_2(1-t, y)}{q^{(b)}_2(1-s, x)}dy.
\end{align*}
\end{theorem}

Applying Theorem~\ref{Thm_MainResult2}, 
we can prove the decomposition formula for the distribution of the BES($\delta$) house-moving (Theorem~\ref{Thm_Decomp_flat_Moving}).
Let $t \in (t_1, t_2)$. 
For $w_1 \in C([t_1, t], \mathbb{R})$ and $w_2 \in C([t, t_2], \mathbb{R})$ 
that satisfy $w_1(t) = w_2(t)$, 
we define $w_1 \oplus_t w_2 \in C([t_1,t_2], \mathbb{R})$ by
\begin{align*}
(w_1 \oplus_t w_2)(s) :=
\begin{cases}
 w_1(s), \qquad s\in [t_1, t], \\
 w_2(s), \qquad s\in [t, t_2].
\end{cases}
\end{align*}

\begin{theorem}\label{Thm_Decomp_flat_Moving}
Let $0\leq a<b$. For every bounded continuous function $F$ on $C([0,1], \mathbb{R})$, it holds that
\begin{align*}
E\left[ F(H^{a \to b}) \right]
= \int_0^b 
E\left[ F(r_{[0,t]}^{a \to y}|_{K_{[0, t]}^-(b)}
\oplus_t H_{[t,1]}^{y \to b})\right] 
P\left( H^{a \to b}(t) \in dy \right), \qquad 0<t<1, 
\end{align*}
where 
$r_{[0,t]}^{a \to y}|_{K_{[0, t]}^-(b)}$ and 
$H_{[t,1]}^{y \to b}$
are chosen to be independent, and 
$H_{[t,1]}^{y \to b}$ is
an $\mathbb{R}$-valued continuous Markov process that satisfies
\begin{align*}
r_{[t,1]}^{y \to b}|_{K_{[t,1]}^-(b+\eta)} \xrightarrow{\mathcal{D}} H_{[t,1]}^{y \to b}, 
\qquad \eta \downarrow 0.
\end{align*}
\end{theorem}

As an application of Theorem~\ref{Thm_Decomp_flat_Moving}, 
we show that the BES($\delta$) house-moving does not hit $b$ 
on the time interval $[0,1)$. 

\begin{prop}\label{Prop_ae_Property_of_Moving}
Let $0\leq a<b$. For $0<t<1$, it holds that
\begin{align*}
P\left( \max_{0 \leq u \leq t} H^{a \to b}(u) < b \right) = 1.
\end{align*}
\end{prop}

By using Theorem~\ref{Thm_MainResult2}, 
we can also prove that the distribution of the BES($\delta$) house-moving 
is absolutely continuous with respect to the BES($\delta$) process.

Let $P^X$ denote the measure on $(C([0,1], \mathbb{R}), \mathcal{B}(C([0,1], \mathbb{R})))$ induced by a continuous process $X=\{X(t)\}_{t\in[0, 1]}$. 
In addition, for $0<t<1$, we define 
\[P_t^{X}:=P\circ (\pi_{[0, t]}\circ X)^{-1}.\]
\begin{theorem}\label{Thm_abs_conti}
Let $0\leq a<b$ and $t\in (0, 1)$. Then, we have 
\[\frac{dP_t^{H^{a\to b}}}{dP_t^{R^a}}(w)=\frac{q^{(b)}_2(1-t, w(t))}{q^{(b)}_2(1, a)}\cdot 1_{K_{[0, t]}^-(b)}(w),\quad w\in C([0, t],\mathbb{R}).\]
\end{theorem}

Let $\tau_{a, b}$ denote the first hitting time of the point $b$ by $R^a$:
\[
\tau_{a, b}:=\inf\{r\geq 0~|~ R^a(r)=b\}.
\]
\begin{prop}\label{Prop_MainResult1}
Let $0\leq a<b$. 
The BES($\delta$) house-moving $H^{a \to b}=\{H^{a \to b}(t) \}_{t \in [0,1]}$ satisfies
\begin{align*}
P\left(H^{a\to b}(t)\in dy\right)
&=P\left(R^a(t)\in dy~|~\tau_{a, b}=1\right),\\
P\left(H^{a\to b}(t)\in dy~|~H^{a\to b}(s)=x\right)
&=P\left(R^a(t)\in dy~|~R^a(s)=x, \tau_{a, b}=1\right)\notag
\end{align*}
for $0<s<t<1$ and $x, y\in (0, b)$.
\end{prop}

Finally, we study the sample path properties of 
BES($\delta$) house-moving $H^{a \to b}$ 
and establish the regularity of its sample path. 

\begin{prop}
\label{Prop_Holder_conti_Moving}
For every $\gamma \in (0, \frac{1}{2})$, 
the path of $H^{a \to b}$ $(0\leq a<b)$ on $[0, 1]$ 
is locally H\"{o}lder-continuous with exponent $\gamma$:
\begin{align*}
P\left( \bigcup_{n=1}^{\infty} \left\{ 
\sup_{\substack{t,s \in [0,1]\\0 < |t-s| \leq \frac{1}{n}}} 
\frac{\left| H^{a \to b}(t) - H^{a \to b}(s) \right|}
{\left| t-s \right|^{\gamma}} 
< \infty \right\} \right) = 1.
\end{align*}
\end{prop}

The remainder of this paper is structured as follows. 
In Section~\ref{section_preliminaries}, we introduce some basic facts related to Bessel processes and Bessel bridges, 
and we prove the results for the distribution of the maximal value of the Bessel bridge. 
Section~\ref{section_preliminaries} is also devoted to proving some inequalities that are used in this paper.  
In Section~\ref{section_const_weakconv}, we prove Theorem~\ref{Thm_MainResult2}, 
which gives the construction of the Bessel house-moving as 
the weak limit of conditioned BES($\delta$) bridges. 
In Section~\ref{Section_DecompFormula}, 
we prove the decomposition formula for the distribution of the Bessel house-moving (Theorem~\ref{Thm_Decomp_flat_Moving}) 
and use this formula to prove some results, including Proposition~\ref{Prop_ae_Property_of_Moving}. 
Section~\ref{Section_abso_conti} is devoted to proving the absolute continuity of the distribution of the BES($\delta$) house-moving 
with respect to the BES($\delta$) process (Theorem~\ref{Thm_abs_conti}). 
In Section~\ref{section_const_moving}, we prove Proposition~\ref{Prop_MainResult1} by using the first hitting time of the Bessel process, 
thus giving us the characterization of the Bessel house-moving. 
We show sample path properties of the Bessel house-moving in Section~\ref{Section_Holder} and Section~\ref{section_numerical}. 
Section~\ref{Section_Holder} is devoted to proving the regularity of the sample path of 
the BES($\delta$) house-moving (Proposition~\ref{Prop_Holder_conti_Moving}). 
In Section~\ref{section_numerical}, we show that the BES($3$) house-moving has the space-time reversal property.

\section{Preliminaries}\label{section_preliminaries}

\subsection{Bessel process and Bessel bridge}\label{subsection_BES_process_bridge}

The BES($\delta$) process is a one-dimensional diffusion generated by 
$\mathcal{L}_{\delta}:=\frac{1}{2}\frac{d^2}{dx^2}+\frac{\delta-1}{2x}\frac{d}{dx}$. 
Note that the point $0$ is an entrance boundary for $\delta\geq 2$ ($\nu\geq 0$) and a regular boundary for $0<\delta<2$ ($-1<\nu<0$). 
In the case that $0<\delta<2$, we impose the reflecting boundary condition at $0$.

In addition, for $0\leq a<b$, the BES($\delta$) bridge from $a$ to $b$ on $[0, 1]$ 
is defined by conditioning the BES($\delta$) process from $a$, $R^a=\{R^a(t)\}_{t\geq 0}$, on $R^a(1)=b$. 

For $t> 0$ and $x, y\in (0, \infty)$, we set 
\begin{align*}
&n_t(x) := \frac{1}{\sqrt{2 \pi t}} \exp \left( -\frac{x^2}{2t} \right),\quad
A^{(\nu)}_t(x,y):=n_t(x)n_t(y)I_{\nu} \left(\frac{xy}{t}\right).
\end{align*}

Let $a, b\geq 0$. 
For $0 < s < t$ and $x, y > 0$, we have the transition densities of $R^a$ (\cite[Chapter XI]{bib_RevuzYor}):
\begin{align*}
&P\left(R^a(t)\in dy\right)
= 2\pi y\left(\frac{y}{a}\right)^{\nu}A^{(\nu)}_t(a, y)dy,
\\
&P\left( R^a(t) \in dy~|~ R^a(s) = x \right)
=2\pi y\left(\frac{y}{x}\right)^{\nu}A^{(\nu)}_{t-s}(x, y)dy.
\end{align*}
For $0<s<t<1$ and $x, y>0$, we have the transition densities of the BES($\delta$) bridge $r^{a\to b}$ on $[0, 1]$ (\cite[Chapter XI]{bib_RevuzYor}): 
\begin{align}
P\left( r^{a \to b}(t) \in dy \right)
&=\frac{P\left(R^a(t)\in dy\right)P\left(R^y(1-t)\in db\right)}{P\left(R^a(1)\in db\right)}
\label{BESbridge_density} \\
&=\frac{2\pi y A^{(\nu)}_t(a, y)A^{(\nu)}_{1-t}(y, b)}{A^{(\nu)}_{1}(a, b)}dy,
\nonumber \\
P\left( r^{a \to b}(t) \in dy ~|~ r^{a \to b}(s) = x \right)
&=\frac{P\left(R^x(t-s)\in dy\right)P\left(R^y(1-t)\in db\right)}{P\left(R^x(1-s)\in db\right)}
\label{BESbridge_transdensity} \\
&=\frac{2\pi yA^{(\nu)}_{t-s}(x, y)A^{(\nu)}_{1-t}(y, b)}
{A^{(\nu)}_{1-s}(x, b)}dy.
\nonumber
\end{align}

In the next lemma, we express the joint densities of the Bessel bridge 
and the maximal value of the Bessel process by the maximal values of the Bessel bridge.
\begin{lem}
\label{Lem_Joint_Dist_BES_bridge}
Let $c\geq 0$ and $0\leq a, b \leq c$. For $0<s<t<1$ and $0\leq x, y\leq c$, we have
\begin{align}
&P\left( r^{a \to b}(t) \in dy, M(r^{a \to b}) \leq c \right)
\nonumber \\
&\quad= P\left( M_{[0, t]}(r_{[0,t]}^{a \to y})\leq c \right) 
P\left( M_{[t, 1]}(r_{[t,1]}^{y \to b})\leq c \right) 
P\left( r^{a \to b}(t) \in dy \right),
\label{Joint_Dist_BES_bridge1}\\
&P\left( r^{a \to b}(t) \in dy, r^{a \to b}(s) \in dx, M(r^{a \to b}) \leq c \right) 
\nonumber \\
&\quad = P\left( M_{[0, s]}(r_{[0,s]}^{a \to x})\leq c \right) 
P\left(M_{[s,t]}(r_{[s,t]}^{x \to y}) \leq c \right) 
P\left( M_{[t, 1]}(r_{[t,1]}^{y \to b}) \leq c \right) \label{Joint_Dist_BES_bridge2}\\
&\qquad \times
P\left( r^{a \to b}(t) \in dy, r^{a \to b}(s) \in dx \right).
\notag
\end{align}
\end{lem}
\begin{proof}
First, we prove \eqref{Joint_Dist_BES_bridge1}. 
By the Markov property of $R^a$, we have
\begin{align}
&P\left( r^{a \to b}(t) \in dy, M(r^{a \to b}) \leq c \right)\nonumber \\
&\quad = \frac{P\left( R^a(t) \in dy, M(R^a) \leq c, R^a(1) \in db \right)}{P\left( R^a(1) \in db \right)} \nonumber \\
&\quad = \frac{P\left( R^y(1-t) \in db, M_{1-t}(R^y) \leq c \right) \times P\left( R^a(t) \in dy, M_t(R^a) \leq c \right)}{P\left( R^a(1) \in db \right)} .
\label{JointDist_BESbridge_eq_step1}
\end{align}
Therefore, because 
\begin{align*}
&P\left( R^y(1-t) \in db, M_{1-t}(R^y) \leq c \right)
= P\left( M_{[t, 1]}(r_{[t,1]}^{y \to b})\leq c \right) P\left( R^y(1-t) \in db \right), \\
&P\left( R^a(t) \in dy, M_t(R^a) \leq c \right)
= P\left( M_{[0, t]}(r_{[0,t]}^{a \to y})\leq c \right) P\left( R^a(t) \in dy \right),
\end{align*}
it follows from \eqref{BESbridge_density} and \eqref{JointDist_BESbridge_eq_step1} that
\begin{align*}
&P\left( r^{a \to b}(t) \in dy, M(r^{a \to b}) \leq c \right)\\
&\quad = P\left( M_{[0, t]}(r_{[0,t]}^{a \to y})\leq c\right) P\left( M_{[t, 1]}(r_{[t,1]}^{y \to b})\leq c \right) 
\frac{P\left( R^a(t) \in dy \right) P\left( R^y(1-t) \in db \right)}
{P\left( R^a(1) \in db \right)} \\
&\quad= P\left( M_{[0, t]}(r_{[0,t]}^{a \to y})\leq c \right) P\left( M_{[t, 1]}(r_{[t,1]}^{y \to b}) \leq c \right) P\left( r^{a \to b}(t) \in dy \right),
\end{align*}
which completes the proof. 
In a similar manner to the proof of \eqref{Joint_Dist_BES_bridge1}, 
we can obtain \eqref{Joint_Dist_BES_bridge2}.
\end{proof}

\subsection{Distribution of the maximal value of the Bessel bridge}

In this subsection, we prove the results for the distribution of the maximal value of the Bessel bridge used in this paper.

\begin{lem}\label{Lem_Esti_J}
Let $X>0$. There exist some $\widetilde{C}_{\nu}>0$ and $N_{\nu}\in\mathbb{N}$ such that 
\begin{align*}
&\frac{n\pi}{2}<j_{\nu, n}<2n\pi, \quad
\left|\frac{1}{J_{\nu+1}(j_{\nu, n})}\right|
\leq \pi\sqrt{n}, \\
&\left|\frac{J_{\nu}\left(Xj_{\nu, n}\right)}{J_{\nu+1}(j_{\nu, n})}\right|
\vee
\left|\frac{J_{\nu+1}\left(Xj_{\nu, n}\right)}{J_{\nu+1}(j_{\nu, n})}\right|
\leq \widetilde{C}_{\nu}
\frac{\left(1+X\pi\right)^{\frac{1}{2}}}{X},\quad (n>N_{\nu}).
\end{align*}
\end{lem}
\begin{proof}
According to \eqref{Esti_zeros}, we can find a natural number $N_{\nu}\geq 2$ that satisfies 
\begin{align*}
\frac{n\pi}{2}<j_{\nu, n}<2n\pi, \quad
\left|\frac{1}{J_{\nu+1}(j_{\nu, n})}\right|
\leq \pi\sqrt{n} \quad (n\geq N_{\nu}).
\end{align*}
In addition, by \eqref{esti_J}, for $n\geq N_{\nu}$, the following inequalities hold:
\begin{align*}
\left|\frac{J_{\nu}\left(X j_{\nu, n}\right)}{J_{\nu+1}(j_{\nu, n})}\right|
&\leq C_{\nu}\frac{\left(X j_{\nu, n}\right)^{\nu}}{\left(1+X j_{\nu, n}\right)^{\nu+\frac{1}{2}}} \pi \sqrt{n}
\leq C_{\nu}\frac{\left(2Xn\pi\right)^{\nu}}{\left(1+\frac{Xn\pi}{2} \right)^{\nu+\frac{1}{2}}} \pi \sqrt{n} 
\leq 2^{2\nu+1/2}C_{\nu} \frac{\left(1+X\pi \right)^{1/2}}{X}, \\
\left|\frac{J_{\nu+1}\left(X j_{\nu, n}\right)}{J_{\nu+1}(j_{\nu, n})}\right|
&\leq C_{\nu+1}\frac{\left(X j_{\nu, n}\right)^{\nu+1}}{\left(1+X j_{\nu, n}\right)^{\nu+\frac{3}{2}}}\pi \sqrt{n}
\leq C_{\nu+1}\frac{\left(2Xn\pi\right)^{\nu+1}}{\left(1+\frac{Xn\pi}{2} \right)^{\nu+\frac{3}{2}}} \pi \sqrt{n} 
\leq 2^{2\nu+5/2}C_{\nu+1} \frac{\left(1+X\pi \right)^{1/2}}{X}.
\end{align*}
\end{proof}

\begin{theorem}[\cite{bib_PitmanYor1999} $(20)$]
\label{Thm_max_dist_of_Diffusion_bridge}
Let $0\leq x, y<c$, and $t>0$, and 
let $p(t; x, y)$ be the symmetric transition density of a regular one-dimensional diffusion on $[0, \infty)$ $R=\{R(t)\}_{t\geq 0}$ relative to its speed measure. 
In addition, let $r_{[0, t]}^{x\to y}=\{r_{[0, t]}^{x\to y}(s)\}_{s\in [0, t]}$ denote an $R$-bridge of length $t$ from $x$ to $y$. 
Moreover, let $\varphi_{\lambda}^{\uparrow}$ and $\varphi_{\lambda}^{\downarrow}$ denote the increasing and decreasing solutions of $Au=\lambda u$ for $A$ the infinitesimal generator of $R$, normalized so that
\begin{equation}\label{PitmanYor_18}
\int_0^{\infty}e^{-\lambda t}p(t; x, y)dt=\varphi_{\lambda}^{\uparrow}(x)\varphi_{\lambda}^{\downarrow}(y),
\quad 0\leq x\leq y, \ \lambda>0.
\end{equation}
Then, we have 
\begin{equation}\label{PitmanYor_20}
\int_0^{\infty}e^{-\lambda t}
P\left(M_{[0, t]}(r_{[0, t]}^{x\to y})>c\right)p(t; x, y)dt
=\varphi_{\lambda}^{\uparrow}(y)\varphi_{\lambda}^{\downarrow}(c)\frac{\varphi_{\lambda}^{\uparrow}(x)}{\varphi_{\lambda}^{\uparrow}(c)}.
\end{equation}
\end{theorem}

\begin{remark}
In the case of the BES($\delta$) process, $\varphi_{\lambda}^{\uparrow}$ and $\varphi_{\lambda}^{\downarrow}$ in Theorem~\ref{Thm_max_dist_of_Diffusion_bridge} are given as follows
(\cite[$(23)$]{bib_PitmanYor1999}):
\begin{equation}\label{PitmanYor_23}
\varphi_{\lambda}^{\uparrow}(x)=I_{\nu}\left(\sqrt{2\lambda}x\right)x^{-\nu},\quad
\varphi_{\lambda}^{\downarrow}(x)=K_{\nu}\left(\sqrt{2\lambda}x\right)x^{-\nu},
\quad x\geq 0, \ \lambda>0.
\end{equation}
\end{remark}

\begin{theorem}\label{Thm_max_dist_of_BES_bridge}
Let $c>0$ and $0 \leq s < t$. For $x, y \in (0,c)$, we have 
\begin{align*}
&P\left(M_{[s,t]}(r_{[s,t]}^{x \to y})\leq c\right)
=\frac{1}{\pi A^{(\nu)}_{t-s}(x, y)}
\sum_{n=1}^{\infty}\frac{J_{\nu}\left(xj_{\nu, n}/c\right) J_{\nu}\left(yj_{\nu, n}/c\right)}{c^2 J_{\nu+1}^2(j_{\nu, n})}\exp\left(-\frac{j_{\nu, n}^2}{2c^2}(t-s)\right).
\end{align*}
In addition, for $y\in [0,c)$, we have 
\begin{align*}
P\left(M_{[s,t]}(r_{[s,t]}^{0 \to y})\leq c\right)
&=P\left(M_{[s,t]}(r_{[s,t]}^{y \to 0})\leq c\right)\\
&=\frac{2(t-s)^{\nu+\frac{1}{2}}}{\sqrt{2\pi}n_{t-s}(y)}
\sum_{n=1}^{\infty}\left(\frac{j_{\nu, n}}{cy}\right)^{\nu}\frac{J_{\nu}\left(yj_{\nu, n}/c\right)}{c^2 J_{\nu+1}^2(j_{\nu, n})}\exp\left(-\frac{j_{\nu, n}^2}{2c^{2}}(t-s)\right).
\end{align*}
\end{theorem}
\begin{proof}
The Laplace transform for a function $f$ is denoted by $L(f)$: 
\[L(f)(\lambda):=\int_0^{\infty}e^{-\lambda s}f(s)ds\quad \lambda>0.
\]
For $0\leq x\leq y<c$, by \eqref{PitmanYor_18}, \eqref{PitmanYor_20}, and \eqref{PitmanYor_23}, we have
\begin{align*}
&L\left(P\left(M_{[0, \cdot]}(r_{[0, \cdot]}^{x\to y})\leq c\right)p(\cdot; x, y)\right)(\lambda)\\
&\quad=L\left(\left(1-P\left(M_{[0, \cdot]}(r_{[0, \cdot]}^{x\to y})>c\right)\right)p(\cdot; x, y)\right)(\lambda)\\
&\quad=L\left(p(\cdot; x, y)\right)(\lambda)-L\left(P\left(M_{[0, \cdot]}(r_{[0, \cdot]}^{x\to y})>c\right)p(\cdot; x, y)\right)(\lambda)\\
&\quad=\varphi_{\lambda}^{\uparrow}(x)\varphi_{\lambda}^{\downarrow}(y)
-\varphi_{\lambda}^{\uparrow}(y)\varphi_{\lambda}^{\downarrow}(c)\frac{\varphi_{\lambda}^{\uparrow}(x)}{\varphi_{\lambda}^{\uparrow}(c)}\\
&\quad=(xy)^{-\nu}I_{\nu}(\sqrt{2\lambda}x)
\frac{K_{\nu}(\sqrt{2\lambda}y)I_{\nu}(\sqrt{2\lambda}c)
-I_{\nu}(\sqrt{2\lambda}y)K_{\nu}(\sqrt{2\lambda}c)}{I_{\nu}(\sqrt{2\lambda}c)},\quad \lambda>0.
\end{align*}
Here, note that 
\begin{equation}\label{PY_161}
\frac{I_{\nu}(XC)}{I_{\nu}(C)}\left(I_{\nu}(C)K_{\nu}(YC)-I_{\nu}(YC)K_{\nu}(C)\right)
=2\sum_{n=1}^{\infty}\frac{J_{\nu}(Xj_{\nu, n})J_{\nu}(Yj_{\nu, n})}{J_{\nu+1}^2(j_{\nu, n})(C^2+j_{\nu, n}^2)}
\end{equation}
holds for $0\leq X\leq Y\leq 1$ and $C>0$ (\cite[$(161)$]{bib_PitmanYor1999}). 
Since we apply this equality for $C=\sqrt{2\lambda}c$, $X=x/c$, and $Y=y/c$, it follows that 
\begin{align}
&L\left(P\left(M_{[0, \cdot]}(r_{[0, \cdot]}^{x\to y})\leq c\right)p(\cdot; x, y)\right)(\lambda)
\nonumber \\
&\quad=
2(xy)^{-\nu}
\sum_{n=1}^{\infty}\frac{J_{\nu}\left(xj_{\nu, n}/c\right) J_{\nu}\left(yj_{\nu, n}/c\right)}{J_{\nu+1}^2(j_{\nu, n})(2\lambda c^2+j_{\nu, n}^2)}
\nonumber \\
&\quad=(xy)^{-\nu}
\sum_{n=1}^{\infty}\frac{J_{\nu}\left(xj_{\nu, n}/c\right) J_{\nu}\left(yj_{\nu, n}/c\right)}{c^2 J_{\nu+1}^2(j_{\nu, n})}\int_0^{\infty}
\exp\left(-\left(\lambda +\frac{j_{\nu, n}^2}{2c^2}\right)r\right)dr
\nonumber \\
&\quad=\sum_{n=1}^{\infty}
\int_0^{\infty}e^{-\lambda r}(xy)^{-\nu}
\frac{J_{\nu}\left(xj_{\nu, n}/c\right) J_{\nu}\left(yj_{\nu, n}/c\right)}{c^2 J_{\nu+1}^2(j_{\nu, n})}\exp\left(-\frac{j_{\nu, n}^2}{2c^2}r\right)dr,
\quad \lambda>0.
\label{Lap_sum_int}
\end{align}
For $n$, we set 
\[f_n(r):=\frac{J_{\nu}\left(xj_{\nu, n}/c\right) J_{\nu}\left(yj_{\nu, n}/c\right)}{c^2 J_{\nu+1}^2(j_{\nu, n})}\exp\left(-\left(\lambda +\frac{j_{\nu, n}^2}{2c^2}\right)r\right),
\quad r\geq 0.\]
Then, by Lemma~\ref{Lem_Esti_J} and \eqref{Esti_zeros}, there exist some $\widetilde{C}_{\nu}>0$ and $N_{\nu}\in\mathbb{N}$ such that 
\begin{align*}
|f_n(r)|&\leq \widetilde{C}_{\nu}^2
\frac{\sqrt{(1+\frac{x\pi}{c})(1+\frac{y\pi}{c})}}{xy}
\exp\left(-\frac{(n\pi)^2}{8c^2}r\right) 
\leq \widetilde{C}_{\nu}^2
\frac{1+\pi}{xy}
\exp\left(-\frac{(n\pi)^2}{8c^2}r\right),\quad n>N_{\nu}.
\end{align*}
Therefore, we can see that
\begin{align*}
&\sum_{n=N_{\nu}+1}^{\infty}\int_0^{\infty}|f_n(r)|dr
\leq
\widetilde{C}_{\nu}^2\frac{1+\pi}{xy}\sum_{n=N_{\nu}+1}^{\infty}\frac{8c^2}{(n\pi)^2}
<\infty
\end{align*}
holds and we can integrate term by term in \eqref{Lap_sum_int}. 
Hence, it follows that 
\begin{align*}
&L\left(P\left(M_{[0, \cdot]}(r_{[0, \cdot]}^{x\to y})\leq c\right)p(\cdot; x, y)\right)(\lambda)\\
&\quad=\int_0^{\infty}e^{-\lambda r}(xy)^{-\nu}
\sum_{n=1}^{\infty}\frac{J_{\nu}\left(xj_{\nu, n}/c\right) J_{\nu}\left(yj_{\nu, n}/c\right)}{c^2 J_{\nu+1}^2(j_{\nu, n})}\exp\left(-\frac{j_{\nu, n}^2}{2c^2}r\right)dr\\
&\quad=L\left((xy)^{-\nu}
\sum_{n=1}^{\infty}\frac{J_{\nu}\left(xj_{\nu, n}/c\right) J_{\nu}\left(yj_{\nu, n}/c\right)}{c^2 J_{\nu+1}^2(j_{\nu, n})}\exp\left(-\frac{j_{\nu, n}^2}{2c^2}(\cdot)\right)\right)(\lambda)\quad (\lambda>0).
\end{align*}
By the inverse Laplace transform of this identity, we obtain the following expression: 
\begin{align*}
P\left(M_{[0, t]}(r_{[0, t]}^{x\to y})\leq c\right)
&=\frac{(xy)^{-\nu}}{\pi(xy)^{-\nu}A^{(\nu)}_t(x, y)}
\sum_{n=1}^{\infty} 
\frac{J_{\nu}\left(xj_{\nu, n}/c\right) J_{\nu}\left(yj_{\nu, n}/c\right)}{c^2 J_{\nu+1}^2(j_{\nu, n})}
\exp\left(-\frac{j_{\nu, n}^2}{2c^2}t\right)\\
&=\frac{1}{\pi A^{(\nu)}_t(x, y)}
\sum_{n=1}^{\infty}
\frac{J_{\nu}\left(xj_{\nu, n}/c\right) J_{\nu}\left(yj_{\nu, n}/c\right)}{c^2 J_{\nu+1}^2(j_{\nu, n})}
\exp\left(-\frac{j_{\nu, n}^2}{2c^2}t\right).
\end{align*}
Because the right-hand side of \eqref{PY_161} is symmetric for $X$ and $Y$, we can see that this result holds for $0< y\leq x<c$.

Finally, for $0\leq y <c$, we can calculate the following: 
\begin{align*}
P\left(M_{[0, t]}(r_{[0, t]}^{0\to y})\leq c\right)
&=\frac{2t^{\nu+\frac{1}{2}}}{\sqrt{2\pi}n_t(y)}
\sum_{n=1}^{\infty}\left(\frac{j_{\nu, n}}{cy}\right)^{\nu}\frac{J_{\nu}\left(yj_{\nu, n}/c\right)}{c^2 J_{\nu+1}^2(j_{\nu, n})}\exp\left(-\frac{j_{\nu, n}^2}{2c^{2}}t\right).
\end{align*}
\end{proof}

\begin{prop}\label{Prop_Diff_max_dist_of_BES_bridge}
Let $\eta>0$ and $0 \leq s < t$. For $x, y \in (0,\eta)$, we have
\begin{align*}
&\frac{\partial}{\partial \eta}P\left(M_{[s,t]}(r_{[s,t]}^{x \to y})\leq \eta\right)\\
&\quad =\frac{1}{\pi A^{(\nu)}_{t-s}(x, y)}
\sum_{n=1}^{\infty}
\frac{1}{J_{\nu+1}^2(j_{\nu, n})}\exp\left(-\frac{j_{\nu, n}^2}{2\eta^2}(t-s)\right)\\
&\quad \qquad \qquad \qquad \qquad 
\times 
\left\{\left(-\frac{2\nu+2}{\eta^3}+\frac{j_{\nu, n}^2}{\eta^5}(t-s)\right)J_{\nu}\left(xj_{\nu, n}/\eta\right)J_{\nu}\left(yj_{\nu, n}/\eta\right)
\right.\\
&\quad \qquad \qquad \qquad \qquad \qquad
\left.+\frac{xj_{\nu, n}}{\eta^4}J_{\nu+1}\left(xj_{\nu, n}/\eta\right)J_{\nu}\left(yj_{\nu, n}/\eta\right)
+\frac{yj_{\nu, n}}{\eta^4}J_{\nu+1}\left(yj_{\nu, n}/\eta\right)J_{\nu}\left(xj_{\nu, n}/\eta\right) \right\},\\
&\frac{\partial}{\partial \eta}P\left(M_{[s,t]}(r_{[s,t]}^{0 \to y})\leq \eta\right)\\
&\quad
=\frac{2(t-s)^{\nu+\frac{1}{2}}}{\sqrt{2\pi }n_{t-s}(y)}
\sum_{n=1}^{\infty}
y^{-\nu}\frac{\left(j_{\nu, n}/\eta\right)^{\nu+2}}{J_{\nu+1}^2(j_{\nu, n})}
\frac{1}{\eta^{3}}\exp\left(-\frac{j_{\nu, n}^2}{2\eta^{2}}(t-s)\right)\\
&\qquad \qquad \qquad \qquad \quad \times 
\left\{\left(t-s-\frac{2\eta^2(\nu+1)}{j_{\nu, n}^2}\right)J_{\nu}\left(yj_{\nu, n}/\eta\right)
+\frac{y\eta}{ j_{\nu, n}}J_{\nu+1}\left(yj_{\nu, n}/\eta\right)\right\}.
\end{align*}
\end{prop}
\begin{proof}
Let $\eta>0$ and let $0<x,y<\eta$ be fixed. 
For $n$, we set 
\begin{align*}
f_n(\eta, x, y)
&=\frac{1}{J_{\nu+1}^2(j_{\nu, n})}\exp\left(-\frac{j_{\nu, n}^2}{2\eta^2}(t-s)\right)
\times 
\left\{\left(-\frac{2\nu+2}{\eta^3}+\frac{j_{\nu, n}^2}{\eta^5}(t-s)\right)J_{\nu}\left(xj_{\nu, n}/\eta\right)J_{\nu}\left(yj_{\nu, n}/\eta\right)\right.\\
&\qquad \qquad \qquad \qquad
\left.+\frac{xj_{\nu, n}}{\eta^4}J_{\nu+1}\left(xj_{\nu, n}/\eta\right)J_{\nu}\left(yj_{\nu, n}/\eta\right)
+\frac{yj_{\nu, n}}{\eta^4}J_{\nu+1}\left(yj_{\nu, n}/\eta\right)J_{\nu}\left(xj_{\nu, n}/\eta\right) \right\}.
\end{align*}
By Lemma~\ref{Lem_Esti_J}, there exist some $\widetilde{C}_{\nu}>0$ and $N_{\nu}\in\mathbb{N}$ such that 
\begin{align*}
&\left|f_n(\eta, x, y)\right|\\
&\leq \exp\left(-\frac{\pi^2(t-s)}{8 \eta^2}n^2 \right)
\times 
\left\{\left(\frac{2\nu+2}{\eta^3}+\frac{(2n\pi)^2}{\eta^5}(t-s)\right)
\eta^2\widetilde{C}_{\nu}^2\frac{\left((1+\frac{x\pi}{\eta})(1+\frac{y\pi}{\eta})\right)^{1/2}}{xy}\right.\\
&\left. \qquad \quad
+\frac{x(2n\pi)}{\eta^4}\eta\widetilde{C}_{\nu}\frac{\left(1+\frac{x\pi}{\eta}\right)^{1/2}}{x}
\cdot\eta \widetilde{C}_{\nu}\frac{\left(1+\frac{y\pi}{\eta}\right)^{1/2}}{y}
+\frac{y(2n\pi)}{\eta^4}\eta\widetilde{C}_{\nu}\frac{\left(1+\frac{y\pi}{\eta}\right)^{1/2}}{y}
\cdot\eta \widetilde{C}_{\nu}\frac{\left(1+\frac{x\pi}{\eta}\right)^{1/2}}{x}\right\} \\
&\leq \exp\left(-\frac{\pi^2(t-s)}{8 \eta^2}n^2 \right)
\times 
\left\{ \frac{2\nu+2}{\eta}+\frac{(2n\pi)^2}{\eta^3}(t-s)
+\frac{2\pi(x+y)}{\eta^2 } n
\right\}
\frac{\widetilde{C}_{\nu}^2 (1+\pi)}{xy},
\quad n>N_{\nu}.
\end{align*}
Therefore, we can differentiate term by term the first identity of Theorem~\ref{Thm_max_dist_of_BES_bridge} in some neighborhood of $\eta$. 
Similarly, for $n$, we set 
\begin{align*}
f_n(\eta, y)
&=\frac{\left(j_{\nu, n}/\eta\right)^{\nu+2}}{J_{\nu+1}^2(j_{\nu, n})}
\frac{1}{\eta^{3}}\exp\left(-\frac{j_{\nu, n}^2}{2\eta^{2}}(t-s)\right)\\
&\quad \times
\left\{\left(t-s-\frac{2\eta^2(\nu+1)}{j_{\nu, n}^2}\right)J_{\nu}\left(yj_{\nu, n}/\eta\right)
+\frac{y\eta}{ j_{\nu, n}}J_{\nu+1}\left(yj_{\nu, n}/\eta\right)\right\}.
\end{align*}
By Lemma~\ref{Lem_Esti_J}, there exist some $\widetilde{C}_{\nu}>0$ and $N_{\nu}\in\mathbb{N}$ such that 
\begin{align*}
\left|f_n(\eta, y)\right|
&\leq \frac{\sqrt{\pi (1+\pi)}}{y\eta^{\nu+4}\sqrt{2}}
\widetilde{C}_{\nu}
(2n\pi)^{\nu+\frac{5}{2}}
\left\{t-s+\frac{8\eta^2(\nu+1)}{(n\pi)^2}
+\frac{2y\eta}{n\pi}\right\}
\exp\left(-\frac{\pi^2(t-s)}{8 \eta^2}n^2 \right),
\ n>N_{\nu}.
\end{align*}
Therefore, we can differentiate term by term the second identity of Theorem~\ref{Thm_max_dist_of_BES_bridge} in some neighborhood of $\eta$. 
\end{proof}

According to Proposition~\ref{Prop_Diff_max_dist_of_BES_bridge} and Lebesgue's dominated convergence theorem, we can obtain the next corollary.
\begin{cor}\label{Cor_Diff_max_dist_of_BES_bridge}
Let $b>0$. For $0 \leq s<t$ and $y \in (0,b)$, we have 
\begin{align*}
\lim_{\eta\downarrow b}\frac{\partial}{\partial \eta}P\left(M_{[s,t]}(r_{[s,t]}^{y \to b})\leq \eta\right)
&=\frac{1}{\pi A^{(\nu)}_{t-s}(y, b)}
\sum_{n=1}^{\infty}\frac{j_{\nu, n}J_{\nu}\left(yj_{\nu, n}/b\right)}{b^3 J_{\nu+1}(j_{\nu, n})}\exp\left(-\frac{j_{\nu, n}^2}{2b^2}(t-s)\right), \\
\lim_{\eta\downarrow b}\frac{\partial}{\partial \eta}P\left(M_{[s,t]}(r_{[s,t]}^{0 \to b})\leq \eta\right)
&=\frac{2(t-s)^{\nu+\frac{1}{2}}}{\sqrt{2\pi }n_{t-s}(b)}
\sum_{n=1}^{\infty}\frac{j_{\nu, n}^{\nu+1}}{b^{2\nu+3}J_{\nu+1}(j_{\nu, n})}\exp\left(-\frac{j_{\nu, n}^2}{2b^{2}}(t-s)\right).
\end{align*}
\end{cor}

\subsection{Some inequalities}

We prepare the following inequalities: 
\begin{lem}\label{Lem_Esti_q_first}
Let $b>0$. There exists some $C_{\nu, b}>0$ such that 
\begin{align}
&
q^{(b+\eta)}_1(s, x, t, y) 
\leq \frac{C_{\nu, b}}{(t-s)^{\nu+1}}n_{t-s}(y-x),
\quad 0\leq s<t\leq 1, \ x, y\in [0, b+\eta), \ \eta\in (0, 1],
\label{Lem_2_2_eq1}\\
&
q^{(b+\eta)}_1(r, z, 1, b) 
\leq 
\frac{C_{\nu, b}}{(1-r)^{\nu+1}}
\left(1\wedge \frac{2\eta(b+\eta)}{1-r}\right) n_{1-r}(z-b),
\quad 0<r<1, \ z\in (0, b+\eta), \ \eta\in (0, 1].
\label{Lem_2_2_eq2}
\end{align}
\end{lem}
\begin{proof}
First, we prove inequality \eqref{Lem_2_2_eq1}. 
By \eqref{esti_I}, there exists some $C_{\nu}>0$ such that 
\[\left(\frac{xy}{t-s}\right)^{-\nu}I_{\nu}\left(\frac{xy}{t-s}\right)
\leq \frac{C_{\nu}}{(1+\frac{xy}{t-s})^{\nu+\frac{1}{2}}}\exp\left(\frac{xy}{t-s}\right).\]
Thus, by this inequality, 
it follows that 
\begin{align*}
q^{(b+\eta)}_1(s, x, t, y)
&\leq \frac{P\left(R^x(t-s)\in dy\right)}{dy}\\
&\leq 2\pi y^{1+\nu}\left(\frac{y}{t-s}\right)^{\nu}n_{t-s}(x) n_{t-s}(y)
\frac{C_{\nu}}{(1+\frac{xy}{t-s})^{\nu+\frac{1}{2}}}\exp \left(\frac{xy}{t-s}\right) \\
&\leq\frac{\widehat{C}_{\nu, b}}{(t-s)^{\nu+1}}n_{t-s}(y-x), 
\end{align*}
where $\widehat{C}_{\nu, b}:=\sqrt{2\pi}C_{\nu}(b+1)^{2\nu+1}\sqrt{1+(b+1)^2}$.
Next, we prove inequality \eqref{Lem_2_2_eq2}. 
According to \cite{bib_FourierBessel}, there exists some $C'_{\nu}>0$ such that 
\begin{align*}
&2(xy)^{-\nu}\sum_{n=1}^{\infty}\frac{J_{\nu}(j_{\nu, n}x)J_{\nu}(j_{\nu,
 n}y)}{J_{\nu+1}^2(j_{\nu, n})}\exp\left(-j_{\nu, n}^2t\right)\\
&\leq C'_{\nu}\frac{(1+t)^{\nu+2}}{(t+xy)^{\nu+1/2}}
\left(1\wedge \frac{(1-x)(1-y)}{t}\right)
\frac{1}{\sqrt{t}}
\exp \left( -\frac{(x-y)^2}{4t}-j_{\nu, 1}^2t\right), 
\quad x, y \in (0,1), \ t>0.
\end{align*}
Using this inequality and Theorem~\ref{Thm_max_dist_of_BES_bridge}, 
we can obtain the following estimate:
\begin{align*}
&q^{(b+\eta)}_1(r, z, 1, b)\\
&\quad=\frac{P\left(R^z(1-r)\in db\right)}{db}
P\left(M_{[r,1]}(r_{[r,1]}^{z\to b})\leq b+\eta\right) \\
&\quad=\frac{b^{2\nu+1}}{(b+\eta)^{2\nu+2}}2\left(\frac{z}{b+\eta}\frac{b}{b+\eta}\right)^{-\nu}
\sum_{n=1}^{\infty}
\frac{J_{\nu}\left(zj_{\nu, n}/(b+\eta)\right) J_{\nu}\left(bj_{\nu, n}/(b+\eta)\right)}{J_{\nu+1}^2(j_{\nu, n})}
\exp\left(-j_{\nu, n}^2\frac{1-r}{2(b+\eta)^2}\right)\\
&\quad \leq 
C'_{\nu}\frac{b^{2\nu+1}}{(b+\eta)^{2\nu+2}}
\left(1+\frac{1-r}{2(b+\eta)^2}\right)^{\nu+2}
\frac{(\frac{(1-r)+2bz}{2(b+\eta)^2})^{1/2}}{(\frac{(1-r)+2bz}{2(b+\eta)^2})^{\nu+1}}
\left(1\wedge \frac{2\eta(b+\eta-z)}{1-r}\right)\\
&\qquad \times
\exp\left(-j_{\nu, 1}^2\frac{1-r}{2(b+\eta)^2}\right)
2\sqrt{\pi}(b+\eta)n_{1-r}(z-b)\\
&\quad \leq 
C'_{\nu}
\frac{b^{2\nu+1}}{(b+\eta)^{2\nu+2}}
\left(\frac{2(b+1)^2+1}{2(b+\eta)^2}\right)^{\nu+2}
\frac{(\frac{1+2(b+1)^2}{2(b+\eta)^2})^{1/2}}{(\frac{1-r}{2(b+\eta)^2})^{\nu+1}}
\left(1\wedge \frac{2\eta(b+\eta)}{1-r}\right)
2\sqrt{\pi}(b+\eta)n_{1-r}(z-b)\\
&\quad \leq 
\frac{C'_{\nu, b}}{(1-r)^{\nu+1}}
\left(1\wedge \frac{2\eta(b+\eta)}{1-r}\right)
n_{1-r}(z-b),
\end{align*}
where $C'_{\nu, b}:=C'_{\nu}
\sqrt{\frac{\pi}{2}}
\frac{(1+2(b+1)^2)^{\nu+5/2}}{b^3}$. 
Since we set $C_{\nu, b}:=\widehat{C}_{\nu, b}\vee C'_{\nu, b}$, we can obtain our assertions.
\end{proof}

\begin{lem}\label{Lem_q2_is_positive}
Let $b>0$. For $0<t\leq1$ and $y\in[0, b)$, we have 
\[
q^{(b)}_2(t,y)>0.\]
\end{lem}
\begin{proof}
According to \cite[Theorem 3.3]{bib_Serafin}, for all $x\in[0, 1)$ and $t>0$, there exists a constant $\ddot{C}_{\nu}>0$ such that 
\[
x^{-\nu}\sum_{n=1}^\infty \frac{j_{\nu, n}J_{\nu}(j_{\nu, n}x)}{J_{\nu+1}(j_{\nu, n})}\exp\left(-\frac{j_{\nu, n}^2}{2}t\right)
\geq \ddot{C}_{\nu}\frac{(1-x)(1+t)^{\nu+2}}{(x+t)^{\nu+\frac{1}{2}}t^{\frac{3}{2}}}
\exp\left(-\frac{(1-x)^2}{2t}-\frac{1}{2}j_{\nu, 1}^2t\right).\]
Hence, by 
Corollary~\ref{Cor_Diff_max_dist_of_BES_bridge}, 
we can prove the assertion as follows: 
\begin{align*}
q^{(b)}_2(t, y)
&=
2\left(\frac{b}{y}\right)^{\nu}
\sum_{n=1}^{\infty}\frac{j_{\nu, n}J_{\nu}\left(yj_{\nu, n}/b\right)}{b^2 J_{\nu+1}(j_{\nu, n})}\exp\left(-\frac{j_{\nu, n}^2}{2b^2}t\right)\\
&\geq \ddot{C}_{\nu} \frac{2}{b^2}
\frac{(1-\frac{y}{b})(1+\frac{t}{b^2})^{\nu+2}}{(\frac{y}{b}+\frac{t}{b^2})^{\nu+\frac{1}{2}}(\frac{t}{b^2})^{\frac{3}{2}}}
\exp\left(-\frac{(b-y)^2}{2t}-\frac{j_{\nu, 1}^2}{2b^2}t\right)>0.
\end{align*}
\end{proof}

\section{Proof of Theorem~\ref{Thm_MainResult2}}\label{section_const_weakconv}

In this section, we prove Theorem~\ref{Thm_MainResult2}, which gives the construction of the Bessel house-moving as 
the weak limit of the conditioned BES($\delta$) bridges. 

\begin{lem}\label{Lem_Densities_Sequences}
Let $0\leq a<b$ and $\eta>0$. For $0 < s < t < 1$ and $x, y \in (0,b+\eta)$, we have
\begin{align}
&P\left( r^{a \to b}|_{K^-(b+\eta)}(t) \in dy \right)
=\frac{q^{(b+\eta)}_1(0, a, t, y)q^{(b+\eta)}_1(t, y, 1, b)}{q^{(b+\eta)}_1(0, a, 1, b)}dy, 
\label{sequence_density_starting_a} \\
&P\left( r^{a \to b}|_{K^-(b+\eta)}(t) \in dy~|~r^{a \to b}|_{K^-(b+\eta)}(s) = x \right) 
=\frac{q^{(b+\eta)}_1(s, x, t, y)q^{(b+\eta)}_1(t, y, 1, b)}{q^{(b+\eta)}_1(s, x, 1, b)}dy.
\label{sequence_transdensity} 
\end{align}
\end{lem}
\begin{proof}
By Lemma~\ref{Lem_Joint_Dist_BES_bridge}, we obtain
\begin{align}
&P\left(r^{a \to b}|_{K^-(b+\eta)}(u) \in dy\right)\nonumber \\
&\quad =\frac{P\left(r^{a \to b}(u) \in dy, M(r^{a \to b})\leq b+\eta\right)}{P\left(M(r^{a \to b}) \leq b+\eta\right)}\nonumber \\
&\quad =\frac{P\left(M_{[0, u]}(r_{[0, u]}^{a \to y})\leq b+\eta\right) P\left(M_{[u, 1]}(r_{[u, 1]}^{y \to b})\leq b+\eta\right)}
{P\left(M(r^{a \to b})\leq b+\eta\right)}
\times P\left(r^{a \to b}(u) \in dy\right), \quad 0<u<1.
\label{CondBESbridge_density}
\end{align}

It holds from \eqref{CondBESbridge_density} and \eqref{BESbridge_density} that
\begin{align*}
&P\left(r^{a \to b}|_{K^-(b+\eta)}(t) \in dy\right)\\
&\quad =\frac{P\left(M_{[0, t]}(r_{[0, t]}^{a \to y})\leq b+\eta\right) P\left(M_{[t, 1]}(r_{[t, 1]}^{y \to b})\leq b+\eta\right)}
{P\left(M(r^{a \to b})\leq b+\eta\right)}
\times \frac{P\left(R^a(t)\in dy\right)P\left(R^y(1-t)\in db\right)}{P\left(R^a(1)\in db\right)}\\
&\quad =\frac{q^{(b+\eta)}_1(0, a, t, y)q^{(b+\eta)}_1(t, y, 1, b) }
{q^{(b+\eta)}_1(0, a, 1, b)}dy.
\end{align*}
Hence, \eqref{sequence_density_starting_a} holds.

Next, we prove \eqref{sequence_transdensity}. By Lemma~\ref{Lem_Joint_Dist_BES_bridge}, we have
\begin{align}
&P\left( r^{a \to b}|_{K^-(b+\eta)}(t) \in dy, r^{a \to b}|_{K^-(b+\eta)}(s) \in dx \right)\notag\\
&\quad=
\frac{P\left(r^{a\to b}(t) \in dy, r^{a\to b}(s) \in dx, M(r^{a\to b})\leq b+\eta\right)}{P\left(M(r^{a\to b})\leq b+\eta\right)}\notag\\
&\quad=
\frac{P\left( M_{[s,t]}(r_{[s,t]}^{x \to y})\leq b+\eta \right) 
P\left( M_{[t, 1]}(r_{[t,1]}^{y \to b})\leq b+\eta \right)P\left( M_{[0, s]}(r_{[0,s]}^{a \to x})\leq b+\eta \right)}
{P\left(M(r^{a\to b})\leq b+\eta\right)}\notag \\
&\qquad \times
P\left( r^{a \to b}(t) \in dy, r^{a \to b}(s) \in dx \right). \label{Eq_Joint_Dist_Cond_Bes_bridges}
\end{align}
Therefore, combining \eqref{CondBESbridge_density}, \eqref{Eq_Joint_Dist_Cond_Bes_bridges} and \eqref{BESbridge_transdensity}, we obtain
\begin{align*}
&P\left( r^{a \to b}|_{K^-(b+\eta)}(t) \in dy~|~r^{a \to b}|_{K^-(b+\eta)}(s) = x \right)\\
&\quad=\frac{P\left( r^{a \to b}|_{K^-(b+\eta)}(t) \in dy, r^{a \to b}|_{K^-(b+\eta)}(s) \in dx \right)}{P\left(r^{a \to b}|_{K^-(b+\eta)}(s) \in dx\right)}\\
&\quad=\frac{
P\left( M_{[s,t]}(r_{[s,t]}^{x \to y})\leq b+\eta \right) 
P\left( M_{[t, 1]}(r_{[t,1]}^{y \to b})\leq b+\eta \right) }
{P\left(M_{[s, 1]}(r_{[s, 1]}^{x \to b})\leq b+\eta\right)}
\times P\left( r^{a \to b}(t) \in dy~|~ r^{a \to b}(s) = x \right)\\
&\quad=\frac{
P\left( M_{[s,t]}(r_{[s,t]}^{x \to y})\leq b+\eta \right) 
P\left( M_{[t, 1]}(r_{[t,1]}^{y \to b})\leq b+\eta \right)}
{P\left(M_{[s, 1]}(r_{[s, 1]}^{x \to b})\leq b+\eta\right)}
\times \frac{P\left(R^x(t-s)\in dy\right)P\left(R^y(1-t)\in db\right)}{P\left(R^x(1-s)\in db\right)}\\
&\quad=\frac{
q^{(b+\eta)}_1(s, x, t, y)
q^{(b+\eta)}_1(t, y, 1, b)
}{
q^{(b+\eta)}_1(s, x, 1, b)}dy.
\end{align*}
Hence, \eqref{sequence_transdensity} holds.
\end{proof}

\begin{prop}
\label{prop_moving_density}
Let $0\leq a<b$. For $0 < s < t < 1$ and $x, y \in (0, b)$, we have
\begin{align}
&\lim_{\eta \downarrow 0}
P\left( r^{a \to b}|_{K^-(b+\eta)}(t) \in dy \right) 
=
\frac{q^{(b)}_1(0, a, t, y)q^{(b)}_2(1-t, y)}{q^{(b)}_2(1, a)}dy,
\label{Eq_Moving_Density}\\
&\lim_{\eta \downarrow 0}
P\left( r^{a \to b}|_{K^-(b+\eta)}(t) \in dy
~\big|~r^{a \to b}|_{K^-(b+\eta)}(s)=x \right) 
=\frac{q^{(b)}_1(s, x, t, y)q^{(b)}_2(1-t, y)}{q^{(b)}_2(1-s, x)}dy. 
\label{Eq_Moving_TransDensity}
\end{align}
\end{prop}
\begin{proof}
By Lemma~\ref{Lem_Densities_Sequences} and L'H\^{o}pital's rule, we obtain our 
assertion.
\end{proof}

Let $b>0$. For $0\leq s<t\leq 1$ and $x, y\in [0, b]$, we define
\begin{equation}\label{Eq_h_b}
h_b(s,x,t,y) 
:= \frac{q^{(b)}_1(s, x, t, y)q^{(b)}_2(1-t, y)}{q^{(b)}_2(1-s, x)}.
\end{equation}

\begin{prop}
\label{Prop_int_h}
Let $b>0$. For $0\leq s<t\leq 1$ and $x\in [0, b)$, we have 
\[\int_0^b h_b(s, x, t, y)dy=1.\]
\end{prop}
\begin{proof}
By \eqref{Eq_h_b}, it suffices to show the following identity: 
\[
q^{(b)}_2(1-s,x)
=\int_0^b q^{(b)}_1(s,x,t,y)q^{(b)}_2(1-t,y)dy.\]
Here, using Lemma~\ref{Lem_Densities_Sequences}, it holds that 
\begin{equation}\label{Eq_KeyLimit}
\frac{q^{(b+\eta)}_1(s, x, 1, b)}{\eta}
=\int_0^{b+\eta}q^{(b+\eta)}_1(s, x, t, y) \frac{q^{(b+\eta)}_1(t, y, 1, b)}{\eta}dy.
\end{equation}
According to L'H\^{o}pital's rule, we obtain 
\begin{equation}\label{Eq_LeftLimit}
\lim_{\eta\downarrow 0}\frac{q^{(b+\eta)}_1(s, x, 1, b)}{\eta}=q^{(b)}_2(1-s, x).
\end{equation}
On the other hand, by Lemma~\ref{Lem_Esti_q_first}, 
for $\eta\in (0, 1)$ and $y\in(0, b+\eta)$, we have the following estimate: 
\begin{align}
q^{(b+\eta)}_1(s, x, t, y) \frac{q^{(b+\eta)}_1(t, y, 1, b)}{\eta}
&\leq \frac{1}{\eta}\frac{C_{\nu, b}}{(t-s)^{\nu+1}}n_{t-s}(y-x)
\frac{C_{\nu, b}}{(1-t)^{\nu+1}}
\left(1\wedge \frac{2\eta(b+\eta)}{1-t}\right)
n_{1-t}(y-b)\notag\\
&\leq \frac{C_{\nu, b}^2(b+1)}{\pi (t-s)^{\nu+3/2}(1-t)^{\nu+5/2}}
<\infty.\label{Ineq_RightLimit}
\end{align}
Again, using L'H\^{o}pital's rule, it holds that 
\begin{equation}\label{Eq_RightLimit}
\lim_{\eta\downarrow 0}q^{(b+\eta)}_1(s, x, t, y) \frac{q^{(b+\eta)}_1(t, y, 1, b)}{\eta}=q^{(b)}_1(s, x, t, y)q^{(b)}_2(1-t, y)
\end{equation}
for $y\in (0, b)$. 
Therefore, by \eqref{Ineq_RightLimit}, \eqref{Eq_RightLimit}, and Lebesgue's dominated convergence theorem, we obtain 
\[\lim_{\eta\downarrow 0}\frac{1}{\eta}\int_0^{b+\eta}q^{(b+\eta)}_1(s, x, t, y) q^{(b+\eta)}_1(t, y, 1, b)dy=\int_0^b q^{(b)}_1(s, x, t, y)q^{(b)}_2(1-t, y)dy.\]
By this equality and \eqref{Eq_LeftLimit}, taking the limit $\eta\downarrow 0$ in \eqref{Eq_KeyLimit} 
allows us to prove the assertion.
\end{proof}

The following proposition implies that $h_b(s, x, t, y)$ satisfies the Chapman--Kolmogorov identity.

\begin{prop}\label{Prop_ChapmanKolmogorov}
Let $b>0$. For $0 < s < t < u < 1$ and $x, z \in (0,b)$, we have 
\[
h_b(s,x,u,z) = \int_0^b h_b(s,x,t,y) \ h_b(t,y,u,z) dy.
\]
\end{prop}
\begin{proof}
By \eqref{Eq_h_b}, it suffices to show the following identity: 
\[
q^{(b)}_1(s,x,u,z) = \int_0^b q^{(b)}_1(s,x,t,y)q^{(b)}_1(t,y,u,z) dy.
\]
According to Lemma~\ref{Lem_Joint_Dist_BES_bridge} 
and \eqref{BESbridge_transdensity}, 
we can prove the assertion as follows:
\begin{align*}
&q^{(b)}_1(s,x,u,z)\\
&\quad=\frac{P\left(R^x(u-s)\in dz\right)}{dz}
\int_0^b P\left(r_{[s, u]}^{x\to z}(t)\in dy, M_{[s, u]}(r_{[s, u]}^{x \to z})\leq b\right)\\
&\quad=\frac{P\left(R^x(u-s)\in dz\right)}{dz}
\int_0^b P\left(M_{[s,t]}(r_{[s,t]}^{x \to y})\leq b\right)P\left(M_{[t, u]}(r_{[t, u]}^{y \to z})\leq b \right) 
P\left(r_{[s, u]}^{x\to z}(t)\in dy\right)\\
&\quad= \int_0^b q^{(b)}_1(s,x,t,y)q^{(b)}_1(t,y,u,z) dy.
\end{align*}
\end{proof}

By Proposition~\ref{Prop_int_h} and Proposition~\ref{Prop_ChapmanKolmogorov}, 
the right sides of \eqref{Eq_Moving_Density} and \eqref{Eq_Moving_TransDensity} 
determine the continuous Markov process $H^{a\to b}=\{H^{a\to b}(t)\}_{t\in [0, 1]}$. 
Then, by Proposition~\ref{prop_moving_density} and 
Lemma~\ref{Ap_Lem_Markov_TransDensityConv_FinDimDistConv}, 
we obtain the convergence $r^{a \to b}|_{K^-(b+\eta)} \to H^{a \to b}$ 
as $\eta \downarrow 0$ in the finite-dimensional distributional sense. 
Therefore, all that remains in proving Theorem~\ref{Thm_MainResult2} 
is the tightness of the family 
$\{ r^{a \to b}|_{K^-(b+\eta)} \}_{0<\eta<\eta_0}$ 
for some $\eta_0>0$. 
By
\[
\lim_{\eta \downarrow 0} q^{(b+\eta)}_2(1, a)  
= q^{(b)}_2(1, a),
\]
we can take $\eta_1 > 0$ so that 
$q^{(b+\eta)}_2(1, a) > q^{(b)}_2(1, a)/2$ holds for $\eta \in (0, \eta_1)$. 
Throughout this section, we fix $\eta_1$ in this fashion and denote 
\begin{equation}\label{def_eta0}
\eta_0 := \min \{ \eta_1, 1 \}.
\end{equation}

\begin{lem}\label{Lem_Esti_q}
Let $0\leq a<b$ and let $0<\eta<\eta_0$ be fixed. We have
\begin{align*}
q^{(b+\eta)}_1(0, a, 1, b) > \eta \frac{q^{(b)}_2(1, a)}{2} .
\end{align*}
\end{lem}
\begin{proof}
According to Taylor's theorem, we can find $\theta \in (0,1)$ so that
\begin{align*}
q^{(b+\eta)}_1(0, a, 1, b) 
=\eta q^{(b+\theta \eta)}_2(1,a) 
>\eta \frac{q^{(b)}_2(1, a)}{2} .
\end{align*}
\end{proof}

Using Lemmas~\ref{Lem_Esti_q_first} and~\ref{Lem_Esti_q}, we obtain the 
following moment inequalities:

\begin{lem}\label{Lem_MomentEq}
Let $0\leq a<b$. 
For each $\alpha>0$, 
we can find a constant $C_{\alpha, \nu, a, b} > 0$ 
such that
\begin{align}
&
\sup_{0 < \eta < \eta_0} 
E\left[ \left| r^{a \to b}|_{K^-(b+\eta)}(r) -r^{a \to b}|_{K^-(b+\eta)}(0)\right|^{2\alpha} \right] 
\leq \frac{C_{\alpha, \nu, a, b}}{r^{\nu+1-\alpha}(1-r)^{\nu+\frac{5}{2}}},
\quad r\in (0, 1), 
\label{Lem_Momentineq_eq1}\\
&
\sup_{0 < \eta < \eta_0} 
E\left[ \left| r^{a \to b}|_{K^-(b+\eta)}(1-r) -r^{a \to b}|_{K^-(b+\eta)}(1)\right|^{2\alpha} \right] 
\leq \frac{C_{\alpha, \nu, a, b}}{r^{\nu+2-\alpha}(1-r)^{\nu+\frac{3}{2}}},
\quad r\in (0, 1),
\label{Lem_Momentineq_eq2}\\
&
\sup_{0 < \eta < \eta_0} 
E\left[ \left| r^{a \to b}|_{K^-(b+\eta)}(t) 
- r^{a \to b}|_{K^-(b+\eta)}(s) \right|^{2\alpha} \right] 
\leq \frac{C_{\alpha, \nu, a, b}}{(t-s)^{\nu+1-\alpha}s^{\nu+\frac{3}{2}}(1-t)^{\nu+\frac{5}{2}}},
\ s, t \in (0, 1).
\label{Lem_Momentineq_eq3}
\end{align}
\end{lem}
\begin{proof}
By Lemmas~\ref{Lem_Esti_q_first} and~\ref{Lem_Esti_q}, we have
\begin{align*}
&P\left(\left.r^{a \to b}\right|_{K^{-}(b+\eta)}(u) \in dz\right)\notag \\
&\quad =\frac{q^{(b+\eta)}_{1}(0, a, u, z) q^{(b+\eta)}_1(u, z, 1, b)}{q^{(b+\eta)}_1(0, a, 1, b)}1_{[0,b+\eta]}(z)dz\\
&\quad \leq 
\frac{2}{\eta q^{(b)}_2(1, a)}
\left(\frac{C_{\nu, b}}{u^{\nu+1}}n_u(z-a)\right)
\left(
\frac{C_{\nu, b}}{(1-u)^{\nu+1}}
\left(1\wedge \frac{2\eta(b+\eta)}{1-u}\right)
n_{1-u}(z-b)\right)dz\\
&\quad \leq 
\frac{4(b+\eta)C_{\nu, b}^2}{q^{(b)}_2(1, a)}
\frac{1}{u^{\nu+1}(1-u)^{\nu+2}}n_u(z-a)n_{1-u}(z-b)dz
\end{align*}
for $0<u<1$. 
On the other hand, for each $c\in\mathbb{R}$, 
\begin{align*}
\int_0^{b+\eta}|z-c|^{2\alpha}n_{r}(z-c)dz
&\leq 2 \int_0^{\infty} w^{2\alpha}n_r(w)dw
=\frac{(2r)^{\alpha}}{\sqrt{\pi}}\Gamma\left(\alpha+\frac{1}{2}\right)
\end{align*}
holds. Hence, because we have
\begin{align*}
&E\left[\left|r^{a \to b}|_{K^{-}(b+\eta)}(r)-r^{a \to b}|_{K^{-}(b+\eta)}(0)\right|^{2\alpha}\right]\notag \\
&\quad \leq 
\frac{4(b+\eta)C_{\nu, b}^2}{\sqrt{2\pi}q^{(b)}_2(1, a)}
\frac{1}{r^{\nu+1}(1-r)^{\nu+5/2}}
\int_{0}^{b+\eta}|z-a|^{2 \alpha} n_{r}(z-a)dz\\
&\quad \leq 
\frac{2^{\alpha+2}(b+\eta)C_{\nu, b}^2 \Gamma\left(\alpha+\frac{1}{2}\right)}{\sqrt{2}\pi q^{(b)}_2(1, a)}
\frac{1}{r^{\nu+1-\alpha}(1-r)^{\nu+5/2}}
\end{align*}
and
\begin{align*}
&E\left[\left|r^{a \to b}|_{K^{-}(b+\eta)}(1-r)-r^{a \to b}|_{K^{-}(b+\eta)}(1)\right|^{2\alpha}\right]\\
&\quad \leq 
\frac{4(b+\eta)C_{\nu, b}^2}{\sqrt{2\pi}q^{(b)}_2(1, a)}
\frac{1}{(1-r)^{\nu+3/2}r^{\nu+2}}
\int_{0}^{b+\eta}|z-b|^{2 \alpha} 
n_{r}(z-b)dz\\
&\quad \leq 
\frac{2^{\alpha+2}(b+\eta)C_{\nu, b}^2 \Gamma\left(\alpha+\frac{1}{2}\right)}{\sqrt{2}\pi q^{(b)}_2(1, a)}
\frac{1}{(1-r)^{\nu+3/2}r^{\nu+2-\alpha}},
\end{align*}
we obtain inequalities \eqref{Lem_Momentineq_eq1} and \eqref{Lem_Momentineq_eq2} as follows:
\begin{align*}
\sup_{0<\eta<\eta_0}
E\left[\left|r^{a \to b}|_{K^{-}(b+\eta)}(r)-r^{a \to b}|_{K^{-}(b+\eta)}(0)\right|^{2\alpha}\right]
&\leq \frac{2^{\alpha+2}(b+1)C_{\nu, b}^2\Gamma\left(\alpha+\frac{1}{2}\right)}{\sqrt{2}\pi q^{(b)}_2(1, a)}
\frac{1}{r^{\nu+1-\alpha}(1-r)^{\nu+5/2}} , \\
\sup_{0<\eta<\eta_0}
E\left[\left|r^{a \to b}|_{K^{-}(b+\eta)}(1-r)-r^{a \to b}|_{K^{-}(b+\eta)}(1)\right|^{2\alpha}\right]
&\leq \frac{2^{\alpha+2}(b+1)C_{\nu, b}^2\Gamma\left(\alpha+\frac{1}{2}\right)}{\sqrt{2}\pi q^{(b)}_2(1, a)}
\frac{1}{r^{\nu+2-\alpha}(1-r)^{\nu+\frac{3}{2}}}.
\end{align*}

Next, we prove \eqref{Lem_Momentineq_eq3}. We note that
\begin{align*}
&P\left(r^{a \to b}|_{K^{-}(b+\eta)}(t) \in dy,\ r^{a \to b}|_{K^{-}(b+\eta)}(s) \in d x\right) \\
&\quad=P\left(r^{a \to b}|_{K^{-}(b+\eta)}(t) \in dy \ | \ r^{a \to b}|_{K^{-}(b+\eta)}(s)=x\right) P\left(r^{a \to b}|_{K^{-}(b+\eta)}(s) \in dx\right)\\
&\quad=\frac{q^{(b+\eta)}_1(0, a, s, x) q^{(b+\eta)}_1(t, y, 1, b)}{q^{(b+\eta)}_1(0, a, 1, b)}q^{(b+\eta)}_1(s, x, t, y)dxdy, 
\qquad 0<x, y<b+\eta .
\end{align*}
By Lemmas~\ref{Lem_Esti_q_first} and \ref{Lem_Esti_q}, we have
\begin{align*}
&P\left(r^{a \to b}|_{K^{-}(b+\eta)}(t) \in dy,\ r^{a \to b}|_{K^{-}(b+\eta)}(s) \in dx\right) \\
&\quad\leq 
\frac{2}{\eta q^{(b)}_2(1, a)}
\cdot
\frac{C_{\nu, b}}{s^{\nu+1}}n_{s}(x-a)
\cdot
\frac{C_{\nu, b}}{(1-t)^{\nu+1}}
\frac{2\eta(b+\eta)}{1-t}n_{1-t}(y-b)
\cdot
\frac{C_{\nu, b}}{(t-s)^{\nu+1}}n_{t-s}(y-x)dxdy\\
&\quad\leq 
\frac{2(b+\eta)C_{\nu, b}^3}{\pi q^{(b)}_2(1, a)}
\cdot
\frac{1}{(t-s)^{\nu+1}s^{\nu+3/2}(1-t)^{\nu+5/2}}
\cdot
n_{t-s}(y-x)dxdy.
\end{align*}
On the other hand, 
\begin{align*}
\iint_{(0, b+\eta)^2}|y-x|^{2 \alpha} n_{t-s}(y-x)dxdy
&=\int_0^{b+\eta}\left(\int_0^{b+\eta} |y-x|^{2 \alpha} n_{t-s}(y-x)dy\right)dx\\
&\leq (b+\eta)\frac{(2(t-s))^{\alpha}}{\sqrt{\pi}}\Gamma\left(\alpha+\frac{1}{2}\right)
\end{align*}
holds. Hence, we obtain 
\begin{align*}
&E\left[\left|r^{a \to b}|_{K^{-}(b+\eta)}(t)-r^{a \to b}|_{K^{-}(b+\eta)}(s)\right|^{2 \alpha}\right] \\
&\quad=\iint_{(0, b+\eta)^2}|y-x|^{2 \alpha} 
P\left(r^{a \to b}|_{K^{-}(b+\eta)}(t) \in dy,\ r^{a \to b}|_{K^{-}(b+\eta)}(s) \in dx\right) \\
&\quad \leq 
\frac{2(b+\eta)C_{\nu, b}^3}{\pi q^{(b)}_2(1, a)}
\cdot
\frac{1}{(t-s)^{\nu+1} s^{\nu+3/2}(1-t)^{\nu+5/2}}
\iint_{(0, b+\eta)^2}|y-x|^{2 \alpha} n_{t-s}(y-x)dxdy\\
&\quad \leq 
\frac{2^{\alpha+1}(b+\eta)^2C_{\nu, b}^3\Gamma\left(\alpha+\frac{1}{2}\right)}{\pi \sqrt{\pi} q^{(b)}_2(1, a)}
\cdot \frac{1}{(t-s)^{\nu+1-\alpha} s^{\nu+3/2}(1-t)^{\nu+5/2}}.
\end{align*}
Therefore, we have
\begin{align*}
&\sup_{0<\eta<\eta_0}
E\left[\left|r^{a \to b}|_{K^{-}(b+\eta)}(t)-r^{a \to b}|_{K^{-}(b+\eta)}(s)\right|^{2 \alpha}\right] \\
&\quad \leq 
\frac{2^{\alpha+1}(b+1)^2C_{\nu, b}^3\Gamma\left(\alpha+\frac{1}{2}\right)}{\pi \sqrt{\pi} q^{(b)}_2(1, a)}
\cdot \frac{1}{(t-s)^{\nu+1-\alpha} s^{\nu+3/2}(1-t)^{\nu+5/2}},
\end{align*}
and inequality \eqref{Lem_Momentineq_eq3} is proved.
\end{proof}

\begin{cor}\label{Cor_restricted_tight}
Let $0\leq a<b$. For each $u\in \left(0, \frac{1}{2}\right)$, the 
family
$\{ \pi_{[u,1-u]} \circ r^{a \to b}|_{K^-(b+\eta)} \}_{\eta\in (0, \eta_0)}$ is tight.
\end{cor}
\begin{proof}
Using inequalities \eqref{Lem_Momentineq_eq1} and \eqref{Lem_Momentineq_eq3} for $\alpha=\frac{1}{2}$ and 
$\alpha=\nu+3$, respectively, we obtain
\begin{align*}
&\sup_{0 < \eta < \eta_0} 
E\left[ \left| r^{a \to b}|_{K^-(b+\eta)}(u) \right|\right]\\
&\quad \leq \sup_{0 < \eta < \eta_0}
\left(E\left[ \left| r^{a \to b}|_{K^-(b+\eta)}(u) -r^{a \to b}|_{K^-(b+\eta)}(0)\right| \right]+E\left[ \left| r^{a \to b}|_{K^-(b+\eta)}(0)\right|\right] \right)\\
&\quad \leq C_{1/2, \nu, a, b}(1-u)^{-\nu-\frac{5}{2}}u^{-\nu-\frac{1}{2}}+a<\infty
\end{align*}
and
\begin{align*}
\sup_{0 < \eta < \eta_0} 
E\left[ \left| r^{a \to b}|_{K^-(b+\eta)}(t) 
- r^{a \to b}|_{K^-(b+\eta)}(s) \right|^{2(\nu+3)} \right] 
&\leq C_{\nu+3, \nu, a, b}s^{-\nu-\frac{3}{2}}(1-t)^{-\nu-\frac{5}{2}}(t-s)^{2}\\
&\leq C_{\nu+3, \nu, a, b}u^{-2\nu-4}(t-s)^{2}
\end{align*}
for $u\leq s<t\leq 1-u$. 
Hence, by Lemma~\ref{Ap_Lem_Moment_estimate_type_suf_cond_for_tightness}, 
we establish the assertion.
\end{proof}

\begin{prop}\label{Prop_Limit_LeftRight}
Let $0\leq a<b$. For $\xi >0$, we have
\begin{align*}
&\lim_{u \downarrow 0} \sup_{\eta\in (0, \eta_0)} 
P\left( \sup_{0 \leq t \leq u} 
|r^{a \to b}|_{K^-(b+\eta)}(t) - r^{a \to b}|_{K^-(b+\eta)}(0)| > \xi \right) = 0,\\
&\lim_{u \downarrow 0} \sup_{\eta\in (0, \eta_0)} 
P\left( \sup_{1-u \leq t \leq 1} 
|r^{a \to b}|_{K^-(b+\eta)}(t) - r^{a \to b}|_{K^-(b+\eta)}(1)| > \xi \right) = 0.
\end{align*}
\end{prop}
\begin{proof}
Applying \eqref{Lem_Momentineq_eq1}, \eqref{Lem_Momentineq_eq2} and \eqref{Lem_Momentineq_eq3} for $\alpha=3\nu+7$ and $t, s, r \in (0,1)$ with $s<t$, we have
\begin{align}
&\sup_{0 < \eta < \eta_0} 
E\left[ \left| r^{a \to b}|_{K^-(b+\eta)}(r) -r^{a \to b}|_{K^-(b+\eta)}(0)\right|^{2(3\nu+7)} \right] 
\leq C_{3\nu+7, \nu, a, b} \frac{r^{2\nu+6}}{(1-r)^{\nu+\frac{5}{2}}}, 
\label{Ineq_momentL} \\
&\sup_{0 < \eta < \eta_0} 
E\left[ \left| r^{a \to b}|_{K^-(b+\eta)}(t) 
- r^{a \to b}|_{K^-(b+\eta)}(s) \right|^{2(3\nu+7)} \right] 
\leq C_{3\nu+7, \nu, a, b} \frac{|t-s|^{2\nu+6}}{s^{\nu+\frac{3}{2}}(1-t)^{\nu+\frac{5}{2}} }, 
\label{Ineq_momentC}\\
&\sup_{0 < \eta < \eta_0} 
E\left[ \left| r^{a \to b}|_{K^-(b+\eta)}(1-r) -r^{a \to b}|_{K^-(b+\eta)}(1)\right|^{2(3\nu+7)} \right] 
\leq C_{3\nu+7, \nu, a, b} \frac{r^{2\nu+5}}{(1-r)^{\nu+\frac{3}{2}}}.
\label{Ineq_momentr}
\end{align}
Let $\gamma = \frac{1}{4\alpha}=\frac{1}{4(3\nu+7)}$, $0<\eta<\eta_0$, 
and fix $n \in \mathbb{N}$. We define
\begin{align*}
&F_n^{\eta} = \left\{ \max_{1 \leq k \leq 2^{n-1}} 
\left| r^{a \to b}|_{K^-(b+\eta)} \left(\frac{k-1}{2^n}\right) 
- r^{a \to b}|_{K^-(b+\eta)} \left(\frac{k}{2^n}\right) \right| 
\geq 2^{-n \gamma} \right\}, \\
&\widetilde{F_n^{\eta}} = \left\{ \max_{2^{n-1} \leq k \leq 2^n} 
\left| r^{a \to b}|_{K^-(b+\eta)} \left(\frac{k-1}{2^n}\right) 
- r^{a \to b}|_{K^-(b+\eta)} \left(\frac{k}{2^n}\right) \right| 
\geq 2^{-n \gamma}  \right\},\\
&a(n, k, \eta) = 
P\left( \left| 
r^{a \to b}|_{K^-(b+\eta)} \left(\frac{k-1}{2^n}\right) 
- r^{a \to b}|_{K^-(b+\eta)} \left(\frac{k}{2^n} \right) 
\right| \geq 2^{-n \gamma} \right), 
\qquad 1 \leq k \leq 2^n.
\end{align*}
Then, by Chebyshev's inequality, we have
\begin{align}
a(n, k, \eta)
&\leq 
2^{\frac{n}{2}}
E\left[ \left| r^{a \to b}|_{K^-(b+\eta)} \left(\frac{k-1}{2^n}\right) 
- r^{a \to b}|_{K^-(b+\eta)} \left(\frac{k}{2^n}\right) \right|^{2(3\nu+7)} \right],
\qquad 1 \leq k \leq 2^n .
\label{Chebyshev_ineq_cond_Besselbridge}
\end{align}
Therefore, using \eqref{Ineq_momentL}, \eqref{Ineq_momentC}, \eqref{Ineq_momentr}, and \eqref{Chebyshev_ineq_cond_Besselbridge}, we have
\begin{align*}
a(n, 1, \eta) 
&\leq 2^{\frac{n}{2}} C_{3\nu+7, \nu, a, b}\left(\frac{2^n}{2^n-1}\right)^{\nu+\frac{5}{2}}
\left( \frac{1}{2^n} \right)^{\alpha-\nu-1} 
\leq C_{3\nu+7, \nu, a, b}2^{-n(\nu+3)} 
\leq C_{3\nu+7, \nu, a, b}2^{-\frac{3}{2}n},\\
a(n, k, \eta) 
&\leq 2^{\frac{n}{2}} C_{3\nu+7, \nu, a, b} \left(\frac{2^n}{k-1}\right)^{\nu+\frac{3}{2}}
\left( \frac{2^n}{2^n-k} \right)^{\nu+\frac{5}{2}} \left( \frac{1}{2^n} \right)^{2\nu+6}
\leq C_{3\nu+7, \nu, a, b} 2^{-\frac{3}{2}n}, 
\quad (2 \leq k \leq 2^n-1),\\
a(n, 2^n, \eta) 
&\leq 2^{\frac{n}{2}} C_{3\nu+7, \nu, a, b}\left(\frac{2^n}{2^n-1}\right)^{\nu+\frac{3}{2}}
\left( \frac{1}{2^n} \right)^{2\nu+5} 
\leq C_{3\nu+7, \nu, a, b} 2^{-n(\nu+3)} 
\leq C_{3\nu+7, \nu, a, b} 2^{-\frac{3}{2}n}.
\end{align*}
Thus, it follows that
\begin{align*}
&P\left( F_n^{\eta} \right) \leq \sum_{k=1}^{2^{n-1}} a(n,k,\eta) 
\leq C_{3\nu+7, \nu, a, b} 2^{-\frac{n}{2}}, \quad
P\left( \widetilde{F}_n^{\eta} \right) \leq \sum_{k=2^{n-1}}^{2^{n}} a(n,k,\eta) 
\leq C_{3\nu+7, \nu, a, b} 2^{-\frac{n}{2}} .
\end{align*}
Therefore, Lemmas~\ref{Ap_Lem_unif_conti_at_t=0} and 
\ref{Ap_Lem_unif_conti_at_t=1} prove the desired results. 
\end{proof}

By Corollary~\ref{Cor_restricted_tight} and Proposition~\ref{Prop_Limit_LeftRight}, we can apply 
Theorem~\ref{Ap_Thm_suf_cond_restricted_tightness} to 
$\{ r^{a \to b}|_{K^-(b+\eta)} \}_{0<\eta<\eta_0}$ 
and obtain the tightness of this family.

\section{Decomposition formula and sample path properties}
\label{Section_DecompFormula}

In this section, we prove the decomposition formula for the distribution of the BES($\delta$) house-moving 
(Theorem~\ref{Thm_Decomp_flat_Moving}). 
In addition, applying this result, we study sample path properties of the BES($\delta$) house-moving. 

First, we prove Theorem~\ref{Thm_Decomp_flat_Moving}. 
By Theorem~\ref{Thm_MainResult2}, because 
$r^{a\to b}|_{K^-(b+\eta)}\xrightarrow{\mathcal{D}} H^{a\to b}$ ($\eta\downarrow 0$) holds, 
\[E\left[F(H^{a\to b})\right]
=\lim_{\eta\downarrow 0}E\left[F(r^{a\to b}|_{K^-(b+\eta)})\right]\]
for every bounded continuous function $F$ on $C([0, 1], \mathbb{R})$. 
We calculate the numerator of
\[
E\left[F(r^{a\to b}|_{K^-(b+\eta)})\right]
=\frac{E\left[F(r^{a\to b})\ ;\ r^{a\to b}\in K^-(b+\eta)\right]}{P\left(r^{a\to b}\in K^-(b+\eta)\right)}\]
as
\begin{align*}
&E\left[F(r^{a\to b})\ ;\ r^{a\to b}\in K^-(b+\eta)\right]\\
&\quad =\int_0^{\infty}
E\left[F(r^{a\to b})\ ;\ r^{a\to b}\in K^-(b+\eta), r^{a\to b}(t)\in dy\right]\\
&\quad =\int_0^{\infty}
E\left[F(r_{[0, t]}^{a\to y}|_{K_{[0, t]}^-(b+\eta)}\oplus_t r_{[t, 1]}^{y\to b}|_{K_{[t, 1]}^-(b+\eta)})\right]
P\left(r^{a\to b}\in K^-(b+\eta), r^{a\to b}(t)\in dy\right)\\
&\quad =\int_0^{\infty}
E\left[F(r_{[0, t]}^{a\to y}|_{K_{[0, t]}^-(b+\eta)}\oplus_t r_{[t, 1]}^{y\to b}|_{K_{[t, 1]}^-(b+\eta)})\right]
P\left(r^{a\to b}|_{K^-(b+\eta)}(t)\in dy\right)P\left(r^{a\to b}\in K^-(b+\eta)\right).
\end{align*}
Hence, we have 
\begin{align*}
E\left[F(r^{a\to b}|_{K^-(b+\eta)})\right]
&=\int_0^{\infty}
E\left[F(r_{[0, t]}^{a\to y}|_{K_{[0, t]}^-(b+\eta)}\oplus_t r_{[t, 1]}^{y\to b}|_{K_{[t, 1]}^-(b+\eta)})\right]
P\left(r^{a\to b}|_{K^-(b+\eta)}(t)\in dy\right).
\end{align*}
Then, it suffices to show that 
\begin{align}
&\lim_{\eta\downarrow 0}
\int_0^{\infty}
E\left[F(r_{[0, t]}^{a\to y}|_{K_{[0, t]}^-(b+\eta)}\oplus_t r_{[t, 1]}^{y\to b}|_{K_{[t, 1]}^-(b+\eta)})\right]
P\left(r^{a\to b}|_{K^-(b+\eta)}(t)\in dy\right)\notag\\
&\quad=\int_0^{\infty}
E\left[F(r_{[0, t]}^{a\to y}|_{K_{[0, t]}^-(b)}\oplus_t H_{[t, 1]}^{y\to b})\right]
P\left(H^{a\to b}(t)\in dy\right). \label{Eq_lim_right}
\end{align}
We obtain the following estimate: 
\begin{align*}
&\left|\int_0^{\infty}
E\left[F(r_{[0, t]}^{a\to y}|_{K_{[0, t]}^-(b+\eta)}\oplus_t r_{[t, 1]}^{y\to b}|_{K_{[t, 1]}^-(b+\eta)})\right]
P\left(r^{a\to b}|_{K^-(b+\eta)}(t)\in dy\right)\right.\\
&\left. \quad -
\int_0^{\infty}E\left[F(r_{[0, t]}^{a\to y}|_{K_{[0, t]}^-(b)}\oplus_t H_{[t, 1]}^{y\to b})\right]
P\left(H^{a\to b}(t)\in dy\right)\right|\\
&\quad \leq
\int_0^{\infty}
\left| E\left[F(r_{[0, t]}^{a\to y}|_{K_{[0, t]}^-(b+\eta)}\oplus_t r_{[t, 1]}^{y\to b}|_{K_{[t, 1]}^-(b+\eta)})
-F(r_{[0, t]}^{a\to y}|_{K_{[0, t]}^-(b)}\oplus_t H_{[t, 1]}^{y\to b})\right] \right| \\
&\qquad \qquad \times P\left(r^{a\to b}|_{K^-(b+\eta)}(t)\in dy\right)\\
&\qquad +\left|\int_0^{\infty}
E\left[F(r_{[0, t]}^{a\to y}|_{K_{[0, t]}^-(b)}\oplus_t H_{[t, 1]}^{y\to b})\right]
\left(P\left(r^{a\to b}|_{K^-(b+\eta)}(t)\in dy\right)-P\left(H^{a\to b}(t)\in dy\right)\right)\right|\\
&\quad =:I_t^{(1)}(\eta)+I_t^{(2)}(\eta).
\end{align*}
Then, if $I_t^{(1)}(\eta), I_t^{(2)}(\eta)\to 0$ as $\eta\downarrow0$, we can prove \eqref{Eq_lim_right}. 
First, consider 
\begin{align*}
I_t^{(1)}(\eta)
&=\int_0^{\infty}\left|
E\left[F(r_{[0, t]}^{a\to y}|_{K_{[0, t]}^-(b+\eta)}\oplus_t r_{[t, 1]}^{y\to b}|_{K_{[t, 1]}^-(b+\eta)})
-F(r_{[0, t]}^{a\to y}|_{K_{[0, t]}^-(b)}\oplus_t H_{[t, 1]}^{y\to b})\right]\right| \\
&\qquad \qquad \times 
\frac{q^{(b+\eta)}_1(0, a, t, y)q^{(b+\eta)}_1(t, y, 1, b)}{q^{(b+\eta)}_1(0, a, 1, b)} 1_{[0,b+\eta]}(y) dy.
\end{align*}
We have 
\begin{align}
&\sup_{\substack{\eta>0\\ y\in (0, b+\eta)}}
\left|E\left[F(r_{[0, t]}^{a\to y}|_{K_{[0, t]}^-(b+\eta)}\oplus_t r_{[t, 1]}^{y\to b}|_{K_{[t, 1]}^-(b+\eta)})
-F(r_{[0, t]}^{a\to y}|_{K_{[0, t]}^-(b)}\oplus_t H_{[t, 1]}^{y\to b})\right]\right| \notag \\
&\quad \leq 2\sup_{w\in C([0,1], \mathbb{R})}\vert F(w)\vert  < \infty .
\label{Ineq_I1_bounded}
\end{align}
By Theorem~\ref{Thm_MainResult2} and Lemma~\ref{Ap_Lem_weakconv_product}, it holds that 
\begin{equation}\label{Eq_conv_product}
\lim_{\eta\downarrow 0}E\left[F(r_{[0, t]}^{a\to y}|_{K_{[0, t]}^-(b+\eta)}\oplus_t r_{[t, 1]}^{y\to b}|_{K_{[t, 1]}^-(b+\eta)})\right]
=E\left[F(r_{[0, t]}^{a\to y}|_{K_{[0, t]}^-(b)}\oplus_t H_{[t, 1]}^{y\to b})\right].
\end{equation}
In addition, by Lemmas~\ref{Lem_Esti_q_first} and~\ref{Lem_Esti_q}, and for $\eta \in (0, \eta_0)$ and $y\in[0, b+\eta]$, we obtain 
\begin{align}
&\frac{q^{(b+\eta)}_1(0, a, t, y)q^{(b+\eta)}_1(t, y, 1, b)}{q^{(b+\eta)}_1(0, a, 1, b)}\notag \\
&\quad \leq \frac{2}{\eta q^{(b)}_2(1, a)}
\left(\frac{C_{\nu, b}}{t^{\nu+1}}n_t(y-a)\right)
\left(\frac{C_{\nu, b}}{(1-t)^{\nu+1}}\left(1\wedge \frac{2\eta(b+\eta)}{1-t}\right)n_{1-t}(y-b)\right)\notag\\
&\quad \leq \frac{2(b+1)C_{\nu, b}^2}{\pi t^{\nu+3/2}(1-t)^{\nu+5/2} q^{(b)}_2(1, a)},
\label{Ineq_cond_bridge_dens}
\end{align}
and by Proposition~\ref{prop_moving_density}, it holds that 
\begin{equation}\label{Eq_cond_bridge_dens}
\lim_{\eta \downarrow 0}
\frac{q^{(b+\eta)}_1(0, a, t, y)q^{(b+\eta)}_1(t, y, 1, b)}{q^{(b+\eta)}_1(0, a, 1, b)}
=
\frac{q^{(b)}_1(0, a, t, y)q^{(b)}_2(1-t, y)}{q^{(b)}_2(1, a)}.
\end{equation}
Therefore, according to \eqref{Ineq_I1_bounded}, \eqref{Eq_conv_product}, \eqref{Ineq_cond_bridge_dens}, \eqref{Eq_cond_bridge_dens}, and Lebesgue's dominated convergence theorem, 
\[I_t^{(1)}(\eta)\to 0, \quad \eta\downarrow0.\]
Next, we consider 
\begin{align*}
I_t^{(2)}(\eta)=
\left\vert
\int_0^{\infty}
E\left[F(r_{[0, t]}^{a\to y}|_{K_{[0, t]}^-(b)}\oplus_t H_{[t, 1]}^{y\to b})\right]\left(P\left(r^{a\to b}|_{K^-(b+\eta)}(t)\in dy\right)-P\left(H^{a\to b}(t)\in dy\right)\right)
\right\vert .
\end{align*} 
We have 
\begin{align*}
\sup_{y>0}
\left|E\left[F(r_{[0, t]}^{a\to y}|_{K_{[0, t]}^-(b)}\oplus_t H_{[t, 1]}^{y\to b})\right]\right|
\leq \sup_{w\in C([0,1], \mathbb{R})}\left|F(w)\right|<\infty .
\end{align*}
Then, by \eqref{Ineq_cond_bridge_dens}, \eqref{Eq_cond_bridge_dens}, and Lebesgue's dominated convergence theorem, 
\begin{align*}
&\lim_{\eta\downarrow 0}
\int_0^{\infty}
E\left[F(r_{[0, t]}^{a\to y}|_{K_{[0, t]}^-(b)}\oplus_t H_{[t, 1]}^{y\to b})\right]
P\left(r^{a\to b}|_{K^-(b+\eta)}(t)\in dy\right)\\
&\quad =\int_0^{\infty}
E\left[F(r_{[0, t]}^{a\to y}|_{K_{[0, t]}^-(b)}\oplus_t H_{[t, 1]}^{y\to b})\right]
P\left(H^{a\to b}(t)\in dy\right).
\end{align*}
Therefore, it follows that 
\[I_t^{(2)}(\eta)\to 0,\quad \eta\downarrow 0.\]
Thus, we prove \eqref{Eq_lim_right} and the proof is completed.
\qed

\begin{lem}\label{Rem_F_canbetakento_Indicator_Moving}
Let $B\in \mathcal{B}(C([0,1], \mathbb{R}))$. 
Then, Theorem~\ref{Thm_Decomp_flat_Moving}
holds true for $F=1_B$. 
\end{lem}
\begin{proof}
Let $A$ be a closed subset of $C([0,1], \mathbb{R})$ 
and let
\begin{align*}
\phi(x):=1-\int_0^1 1_{(-\infty, x]}(u)du, \qquad x\in \mathbb{R}.
\end{align*}
Then, for $w\in C([0,1],\mathbb{R})$, we have
\begin{align*}
&F_n(w):= \phi(nd_{\infty}(w,A))\downarrow 1_A(w), \qquad n\to \infty,
\end{align*}
where 
\[
d_{\infty}(w, A):=\inf\{\|w-v\|_{C([0, 1], \mathbb{R})}\mid v\in A\}.
\]
Therefore, by Lebesgue's dominated convergence theorem 
and Dynkin's $\pi$-$\lambda$ theorem, we can obtain our assertion.
\end{proof}

\begin{lem}\label{30}
Let $0\leq a<b$. For $0<z\leq x \leq b$ and $t \in (0,1)$, 
\begin{align*}
&P\left( \max_{u\in [0,t]}H^{a \to b}(u) = x \right) = 0, \\
&P\left( \max_{u\in [0,t]}H^{a \to b}(u) \leq x, H^{a \to b}(t) \leq z \right)
=\int_0^z \dfrac{q^{(x)}_1(0, a, t,y)q^{(b)}_2(1-t,y)}{q^{(b)}_2(1, a)}dy.
\end{align*}
\end{lem}
\begin{proof}
Let $A_i$ $(i=1,2)$ be closed subsets of $C([0,1], \mathbb{R})$ given by
\begin{align*}
&A_1:=\left\{ w \in C([0,1], \mathbb{R})\ \big\vert \ \max_{u\in [0,t]}w(u)= x \right\}, \\
&A_2:=\left\{ w \in C([0,1], \mathbb{R})
\ \big\vert \ \max_{u\in [0,t]}w(u)\leq x, \ w(t)\leq z \right\}.
\end{align*}
Lemma~\ref{Rem_F_canbetakento_Indicator_Moving} implies that 
Theorem~\ref{Thm_Decomp_flat_Moving} 
can be applied for $F=1_{A_i}$ $(i=1,2)$. 
Thus, we obtain
\begin{align}
&P\left( M_t(H^{a \to b}) = x \right)
=\int_0^x P\left( r_{[0,t]}^{a \to y}|_{K_{[0,t]}^-(b)}\in \partial K_{[0,t]}^-(x) \right) 
P\left( H^{a \to b}(t) \in dy \right), 
\label{flatMoving_Decomp_Mt_eq_x}\\
&P\left( M_t(H^{a \to b}) \leq x, H^{a \to b}(t) \leq z \right)\notag \\
&\quad =\int_0^z 
P\left( r_{[0,t]}^{a \to y}|_{K_{[0,t]}^-(b)}\in K_{[0,t]}^-(x),
\ r_{[0,t]}^{a \to y}|_{K_{[0,t]}^-(b)}(t)\leq z \right) 
P\left( H^{a \to b}(t) \in dy \right).
\label{flatMoving_Decomp_Ht_leq_z_Mt_leq_x}
\end{align}
By Proposition ~\ref{Prop_Diff_max_dist_of_BES_bridge}
and \eqref{flatMoving_Decomp_Mt_eq_x}, we obtain
\begin{align*}
P\left( M_t(H^{a \to b}) = x \right)
=\int_0^x 
\frac{P\left( r_{[0,t]}^{a \to y} \in \partial K_{[0,t]}^-(x) \right)}
{P\left( r_{[0,t]}^{a \to y} \in K_{[0,t]}^-(b)\right)}
P\left( H^{a \to b}(t) \in dy \right)=0.
\end{align*}
Furthermore, 
\eqref{flatMoving_Decomp_Ht_leq_z_Mt_leq_x} implies that
\begin{align*}
&P\left( M_t(H^{a \to b}) \leq x, H^{a \to b}(t) \leq z \right)\notag \\
&\quad =\int_0^z 
\frac{P\left( r_{[0,t]}^{a \to y} \in K_{[0,t]}^-(x),\ r_{[0,t]}^{a \to y}(t)\leq z \right)}
{P\left( r_{[0,t]}^{a \to y} \in K_{[0,t]}^-(b)\right)}
\dfrac{q^{(b)}_1(0, a, t,y)q^{(b)}_2(1-t,y)}{q^{(b)}_2(1, a)} dy\\
&\quad =\int_0^z 
\frac{P\left( r_{[0,t]}^{a \to y} \in K_{[0,t]}^-(x) \right)}
{P\left( r_{[0,t]}^{a \to y} \in K_{[0,t]}^-(b)\right)}
\dfrac{q^{(b)}_1(0, a, t,y)q^{(b)}_2(1-t,y)}{q^{(b)}_2(1, a)}dy\\
&\quad =\int_0^z 
\dfrac{q^{(x)}_1(0, a, t,y)q^{(b)}_2(1-t,y)}{q^{(b)}_2(1, a)}dy .
\end{align*}
\end{proof}

\subsection{Proof of Proposition~\ref{Prop_ae_Property_of_Moving}}

Let $t\in (0,1)$. Lemma~\ref{30} implies that 
\begin{align*}
&P\left( M_t(H^{a \to b}) = b \right) = 0, \\
&P\left( M_t(H^{a \to b}) \leq b \right) 
= P\left( M_t(H^{a \to b}) \leq b, H^{a \to b}(t) \leq b \right) 
= \int_0^b P\left( H^{a \to b}(t) \in dy \right) = 1.
\end{align*}
Therefore, $P( M_t(H^{a \to b}) < b) = 1$
holds and Proposition~\ref{Prop_ae_Property_of_Moving} is obtained. 
Proposition~\ref{Prop_ae_Property_of_Moving} implies that 
Bessel house-moving $H^{a \to b}$ 
does not hit $b$ on the time interval $[0, 1)$.

\section{Proof of Theorem~\ref{Thm_abs_conti}}
\label{Section_abso_conti}
For $t>0$ and $x, y\in[0, \infty)$, we set 
\begin{align*}
p(t; x, y):=\dfrac{P\left(R^x(t)\in dy\right)}{dy}.
\end{align*}
In addition, we denote the expectation with respect to a probability $Q$ by $E^Q$.

First, we prepare two lemmas. 
\begin{lem}
\label{Lem_abs_conti_BES_process_and_BES_bridge}
Let $0\leq a<b$. For $t\in(0, 1)$, we have
\[\frac{dP_t^{r^{a\to b}}}{dP_t^{R^a}}(w)=\frac{p(1-t; w(t), b)}{p(1; a, b)},\quad w\in C([0, t], \mathbb{R}).\]
\end{lem}
\begin{proof}
Let $A\in \mathcal{B}(C([0, t],\mathbb{R}))$ be fixed. 
By the Markov property of $R^a$, we obtain the assertion as follows: 
\begin{align*}
P_t^{r^{a\to b}}(A)
&=\frac{P^{R^a}\left(\pi_{[0, t]}^{-1}(A), w(1)\in db\right)}{P^{R^a}\left(w(1)\in db\right)}\\
&=\frac{E^{R^a}\left[1_{\pi_{[0, t]}^{-1}(A)}(w)\cdot P^{R^a}\left(w(1)\in db~|~w(t)\right)\right]}{P^{R^a}\left(w(1)\in db\right)}\\
&=\int_{\pi_{[0, t]}^{-1}(A)}\frac{P^{R^a}\left(w(1)\in db~|~w(t)\right)}{P^{R^a}\left(w(1)\in db\right)}P^{R^a}\left(dw\right)\\
&=\int_{A}\frac{p(1-t; w(t), b)}{p(1; a, b)}P_t^{R^a}\left(dw\right).
\end{align*}
\end{proof}

\begin{lem}\label{Lem_RN_dens}
Let $0\leq a<b$ and $t\in (0, 1)$. For every bounded continuous functions $F$ on $C([0, t], \mathbb{R})$, it holds that 
\[
E^{P_t^{H^{a\to b}}}[F]=\int_{C([0, t], \mathbb{R})}F(w)\cdot \frac{q^{(b)}_2(1-t, w(t))}{q^{(b)}_2(1, a)}\cdot 1_{K_{[0, t]}^-(b)}(w)P_t^{R^a}(dw).\]
\end{lem}
\begin{proof}
By the Markov property of $r^{a\to b}$ and Lemma~\ref{Lem_abs_conti_BES_process_and_BES_bridge}, for $A\in\mathcal{B}(C([0, t], \mathbb{R}))$, it holds that 
\begin{align*}
&P^{r^{a\to b}}\left(\pi_{[0, t]}^{-1}(A)~|~K^-(b+\eta)\right)\\
&\quad =\frac{P^{r^{a\to b}}\left(\pi_{[0, t]}^{-1}(A)\cap \pi_{[0, t]}^{-1}(K_{[0, t]}^-(b+\eta))\cap \pi_{[t, 1]}^{-1}(K_{[t, 1]}^-(b+\eta))\right)}
{P^{r^{a\to b}}\left(K^-(b+\eta)\right)}\\
&\quad =\int_{\pi_{[0, t]}^{-1}(A)}
\frac{P^{r^{a\to b}}\left(\pi_{[t, 1]}^{-1}(K_{[t, 1]}^-(b+\eta))~|~w(t)\right)}{P^{r^{a\to b}}\left(K^-(b+\eta)\right)}
1_{K_{[0, t]}^-(b+\eta)}(\pi_{[0, t]}\circ w) P^{r^{a\to b}}\left(dw\right)\\
&\quad =\int_{\pi_{[0, t]}^{-1}(A)}
\frac{P^{r^{a\to b}}\left(\pi_{[t, 1]}^{-1}(K_{[t, 1]}^-(b+\eta))~|~w(t)\right)p(1-t; w(t), b)}
{P^{r^{a\to b}}\left(K^-(b+\eta)\right)p(1; a, b)}
1_{K_{[0, t]}^-(b+\eta)}(\pi_{[0, t]}\circ w) P^{R^a}\left(dw\right)\\
&\quad =\int_{A}\frac{q^{(b+\eta)}_1(t, w(t), 1, b)}{q^{(b+\eta)}_1(0, a, 1, b)}1_{K_{[0, t]}^-(b+\eta)}(w)P_t^{R^a}\left(dw\right),\quad \eta>0.
\end{align*}
Then, for a bounded continuous function $F$ on $C([0, t], \mathbb{R})$, we obtain 
\begin{align*}
&\int_{C([0, 1], \mathbb{R})}F(\pi_{[0, t]}\circ w)P^{r^{a\to b}}\left(dw~|~K^-(b+\eta)\right)\\
&\quad=\int_{C([0, t], \mathbb{R})}
F(w)\frac{q^{(b+\eta)}_1(t, w(t), 1, b)}{q^{(b+\eta)}_1(0, a, 1, b)}
1_{K_{[0, t]}^-(b+\eta)}(w)P_t^{R^a}\left(dw\right),\quad \eta>0.
\end{align*}
By Lemmas~\ref{Lem_Esti_q_first} and~\ref{Lem_Esti_q}, and for $\eta \in (0, \eta_0)$, we obtain 
\begin{align*}
&\frac{q^{(b+\eta)}_1(t, w(t), 1, b)}{q^{(b+\eta)}_1(0, a, 1, b)} 
\leq \frac{4(b+\eta)C_{\nu, b}}{\sqrt{2\pi }q^{(b)}_2(1, a)}
\frac{1}{(1-t)^{\nu+5/2}}, \qquad w \in C([0, t], [0,b+\eta]).
\end{align*}
In addition, it holds that
\begin{align*}
\lim_{\eta\downarrow0}\frac{q^{(b+\eta)}_1(t, w(t), 1, b)}{q^{(b+\eta)}_1(0, a, 1, b)}1_{K_{[0, t]}^-(b+\eta)}(w)
=\frac{q^{(b)}_2(1-t, w(t))}{q^{(b)}_2(1, a)}1_{K_{[0, t]}^-(b)}(w), \qquad w\in C([0, t], [0,\infty)).
\end{align*}
Therefore, Lebesgue's dominated convergence theorem implies
\begin{align*}
&\lim_{\eta\downarrow0}
\int_{C([0, 1], \mathbb{R})}F(\pi_{[0, t]}\circ w)P^{r^{a\to b}}\left(dw~|~K^-(b+\eta)\right)\\
&\quad =\int_{C([0, t], \mathbb{R})}F(w)\frac{q^{(b)}_2(1-t, w(t))}{q^{(b)}_2(1, a)}1_{K_{[0, t]}^-(b)}(w)P_t^{R^a}\left(dw\right).
\end{align*}
According to this equality and Theorem~\ref{Thm_MainResult2}, it follows that 
\begin{align*}
E^{P_t^{H^{a\to b}}}[F]
&=\int_{C[0, 1], \mathbb{R})}F(\pi_{[0, t]}\circ w)P^{H^{a\to b}}(dw)\\
&=\lim_{\eta\downarrow0}\int_{C([0, 1], \mathbb{R})}F(\pi_{[0, t]}\circ w)P^{r^{a\to b}}\left(dw~|~K^-(b+\eta)\right)\\
&=\int_{C([0, t], \mathbb{R})}F(w)\frac{q^{(b)}_2(1-t, w(t))}{q^{(b)}_2(1, a)}1_{K_{[0, t]}^-(b)}(w)P_t^{R^a}\left(dw\right).
\end{align*}
Thus, the proof is completed.
\end{proof}

Now, we prove Theorem~\ref{Thm_abs_conti}. 
Let $A$ be a closed subset of $C([0, t], \mathbb{R})$.
In a similar manner to the proof of Lemma~\ref{Rem_F_canbetakento_Indicator_Moving}, 
by Lemma~\ref{Lem_RN_dens} and Lebesgue's dominated convergence theorem, it holds that 
\begin{align}
E^{P_t^{H^{a\to b}}}[1_A]
&=\int_{C([0, t], \mathbb{R})}1_A(w)\frac{q^{(b)}_2(1-t, w(t))}{q^{(b)}_2(1, a)}1_{K_{[0, t]}^-(b)}(w)P_t^{R^a}(dw).
\label{indicator_A_PtHatob}
\end{align}
Using \eqref{indicator_A_PtHatob} and Dynkin's $\pi$-$\lambda$ theorem, we can prove the assertion completely.
\qed


\section{Proof of Proposition~\ref{Prop_MainResult1}}\label{section_const_moving}

In this section, we prove Proposition~\ref{Prop_MainResult1}, 
which gives the characterization of the Bessel house-moving by using the first hitting time of the Bessel process.

\begin{lem}
\label{Lem_relation_fht_q}
Let $b>0$. For $t>0$ and $y\in (0, b)$, we have
\begin{align}
&\frac{P\left(\tau_{y, b}\in dt\right)}{dt}
=\frac{q^{(b)}_2(t, y)}{2}, 
\label{relation_q2_fht}\\
&\frac{P\left(\tau_{0, b}\in dt\right)}{dt}
=\frac{q^{(b)}_2(t, 0)}{2}. 
\label{relation_qbb_fht}
\end{align}
\end{lem}
\begin{proof}
First, we prove \eqref{relation_q2_fht}. It holds that 
\begin{align}
\frac{P\left(\tau_{y, b}\in dt\right)}{dt}
&=-\frac{\partial}{\partial{t}}P\left(M_{t}(R^y)< b\right)\notag\\
&=-\frac{\partial}{\partial{t}}\int_0^bP\left(M_{t}(R^y)\leq b, R^y(t)\in dx\right)\notag\\
&=-\frac{\partial}{\partial{t}}\int_0^bP\left(M_{[0, t]}(r_{[0, t]}^{y\to x})\leq b\right)P\left(R^y(t)\in dx\right).\notag
\end{align}
For each $n$, 
we set 
\[f_n(t, x):=\frac{J_{\nu}\left(xj_{\nu, n}/b\right) J_{\nu}\left(yj_{\nu, n}/b\right)}{b^2 J_{\nu+1}^2(j_{\nu, n})}\exp\left(-\frac{j_{\nu, n}^2}{2b^2}t\right), 
\quad t\geq 0, \ x\in(0, b).\]
Then, by Theorem~\ref{Thm_max_dist_of_BES_bridge}, we have
\[q^{(b)}_1(0, y, t, x)=2x\left(\frac{x}{y}\right)^{\nu}\sum_{n=1}^{\infty}f_n(t, x). \]
Let $T>0$ be fixed. By Lemma~\ref{Lem_Esti_J} and \eqref{Esti_zeros}, there exist some $\widetilde{C}_{\nu}>0$ and $N_{\nu}\in\mathbb{N}$ such that 
\begin{align*}
\left|\frac{\partial}{\partial{t}}f_n(t, x)\right|
&=\left|\frac{J_{\nu}\left(xj_{\nu, n}/b\right) J_{\nu}\left(yj_{\nu, n}/b\right)}{b^2 J_{\nu+1}^2(j_{\nu, n})}\frac{j_{\nu, n}^2}{2b^2}\exp\left(-\frac{j_{\nu, n}^2}{2b^2}t\right)\right|\\
&\leq 2\widetilde{C}_{\nu}^2\frac{\sqrt{(1+\frac{x\pi}{b})(1+\frac{y\pi}{b})}}{xy} (n\pi)^2\exp\left(-\frac{(n\pi)^2}{8b^2}t\right)
\end{align*}
holds for $n>N_{\nu}$ and $t\in (T, \infty)$. 
Since 
\[\sum_{n=N_{\nu}+1}^{\infty}(n\pi)^2\exp\left(-\frac{(n\pi)^2}{8b^2}T\right)<\infty \]
holds, we have
\begin{align*}
\frac{\partial}{\partial{t}}
q^{(b)}_1(0, y, t, x)=2x\left(\frac{x}{y}\right)^{\nu}\sum_{n=1}^{\infty}\frac{\partial}{\partial{t}}f_n(t, x)
\quad t\in (T, \infty), 
\end{align*}
by Lebesgue's dominated convergence theorem. Thus, we obtain 
\begin{align}
&\sup_{t\in(T, \infty)}\left|\frac{\partial}{\partial{t}}q^{(b)}_1(0, y, t, x)\right| \notag \\
&\quad 
\leq y^{-\nu}\sum_{n=1}^{N_{\nu}}\frac{\left|x^{\nu+1}J_{\nu}\left(xj_{\nu, n}/b\right) J_{\nu}\left(yj_{\nu, n}/b\right)\right|}{b^2 J_{\nu+1}^2(j_{\nu, n})}\frac{j_{\nu, n}^2}{b^2}\exp\left(-\frac{j_{\nu, n}^2}{2b^2}T\right)
\notag\\
&\qquad + 4\widetilde{C}_{\nu}^2 \left(\frac{x}{y}\right)^{\nu}\frac{\sqrt{(1+\frac{x\pi}{b})(1+\frac{y\pi}{b})}}{y} 
\sum_{n=N_{\nu}+1}^\infty (n\pi)^2\exp\left(-\frac{(n\pi)^2}{8b^2}T\right),
\quad x\in(0, b).
\label{Check_integrable}
\end{align}
By \eqref{esti_J}, 
because there exists $C_{\nu}>0$ such that 
\begin{align*}
\left|x^{\nu+1}J_{\nu}\left(xj_{\nu, n}/b\right)\right|
&\leq C_{\nu}\left(\frac{j_{\nu, n}}{b}\right)^{\nu}\frac{x^{2\nu+1}}{\left(1+\frac{xj_{\nu, n}}{b}\right)^{\nu+\frac{1}{2}}}
\leq C_{\nu}\left(\frac{j_{\nu, n}}{b}\right)^{-1}x^{\nu}\left(1+\frac{xj_{\nu, n}}{b}\right),
\quad x\in (0, b)
\end{align*}
holds, the functions $x^{\nu+1}J_{\nu}\left(xj_{\nu, n}/b\right), n=1, \ldots, N_{\nu}$ 
in the first term on the right-hand side of \eqref{Check_integrable} are integrable with respect to $x$ on $[0, b]$. 
In addition, since 
\[x^{\nu} \sqrt{1+\frac{x\pi}{b}}\leq x^{\nu}\left(1+\frac{x\pi}{b}\right)\quad(x\in (0, b))\]
holds, the function $x^{\nu}\sqrt{1+\frac{x\pi}{b}}$ in the second term on the right-hand side of \eqref{Check_integrable} is integrable with respect to $x$ on $[0, b]$. 
Therefore, by Lebesgue's dominated convergence theorem,  
\[\frac{\partial}{\partial{t}}\int_0^b q^{(b)}_1(0, y, t, x)dx
=\int_0^b\frac{\partial}{\partial{t}}q^{(b)}_1(0, y, t, x)dx.\]
Recall that $\nu = \delta /2 - 1$, and
let $m(x)dx=2x^{2\nu+1}dx$ be the speed measure of the BES($\delta$) process. Then, we obtain 
\begin{align*}
&m(x)\frac{\partial}{\partial{t}}\left(\frac{q^{(b)}_1(0, y, t, x)}{m(x)}\right)
=m(x)\mathcal{L}_{\delta}\left(\frac{q^{(b)}_1(0, y, t, x)}{m(x)}\right)
=\frac{1}{2}\frac{\partial}{\partial x}\left(m(x)\frac{\partial}{\partial x}\left(\frac{q^{(b)}_1(0, y, t, x)}{m(x)}\right)\right), 
\end{align*}
where $\mathcal{L}_{\delta}$ is the infinitesimal generator of the BES($\delta$) process. 
So, we get 
\begin{align*}
\frac{P\left(\tau_{y, b}\in dt\right)}{dt}
&=-\frac{1}{2}\int_0^b\frac{\partial}{\partial x}\left(m(x)\frac{\partial}{\partial x}\left(\frac{q^{(b)}_1(0, y, t, x)}{m(x)}\right)\right)dx\\
&=-\frac{1}{2}\left[m(x)\frac{\partial}{\partial x}\left(\frac{q^{(b)}_1(0, y, t, x)}{m(x)}\right)\right]_{x=0}^{x=b}.
\end{align*}
Inequality~\eqref{esti_J} and Lemma~\ref{Lem_Esti_J} imply the following inequality: 
\begin{align}
&\sum_{n=1}^{\infty}\sup_{x\in (0, \infty)}\left|\frac{j_{\nu, n}x^{-(\nu+1)}J_{\nu+1}\left(xj_{\nu, n}/b\right)J_{\nu}\left(yj_{\nu, n}/b\right)}{J_{\nu+1}^2(j_{\nu, n})}\exp\left(-\frac{j_{\nu, n}^2}{2b^2}t\right)\right|\notag\\
&\quad \leq C_{\nu+1}\sum_{n=1}^{\infty}
j_{\nu, n}\left(\frac{j_{\nu, n}}{b}\right)^{\nu+1} 
\left|\frac{J_{\nu}\left(yj_{\nu, n}/b\right)}{J_{\nu+1}^2(j_{\nu, n})}\right|
\exp\left(-\frac{j_{\nu, n}^2}{2b^2}t\right) \notag\\
&\quad \leq C_{\nu+1}\sum_{n=1}^{N_{\nu}}
j_{\nu, n}\left(\frac{j_{\nu, n}}{b}\right)^{\nu+1} 
\left|\frac{1}{J_{\nu+1}(j_{\nu, n})}\right|
\left|\frac{J_{\nu}\left(yj_{\nu, n}/b\right)}{J_{\nu+1}(j_{\nu, n})}\right|
\exp\left(-\frac{j_{\nu, n}^2}{2b^2}t\right)\notag \\
&\qquad +\frac{C_{\nu+1}\widetilde{C}_{\nu} \pi \left(1+y\pi/b\right)^{\frac{1}{2}}}{yb^{\nu}}
\sum_{n=N_{\nu}+1}^{\infty}\sqrt{n} (2n\pi)^{\nu+2}
\exp\left(-\frac{(n\pi)^2}{8b^2}t\right). \label{Ineq_Check_x}
\end{align}
Then Lebesgue's dominated convergence theorem and the inequality \eqref{Ineq_Check_x} show that 
\begin{align*}
&m(x)\frac{\partial}{\partial x}\left(\frac{q^{(b)}_1(0, y, t, x)}{m(x)}\right)\\
&\quad 
=2x^{2\nu+1}\frac{\partial}{\partial x}\left(\left(xy\right)^{-\nu}\sum_{n=1}^{\infty}\frac{J_{\nu}\left(xj_{\nu, n}/b\right) J_{\nu}\left(yj_{\nu, n}/b\right)}{b^2 J_{\nu+1}^2(j_{\nu, n})}\exp\left(-\frac{j_{\nu, n}^2}{2b^2}t\right)\right)\\
&\quad 
=-2x^{2\nu+2} y^{-\nu}
\sum_{n=1}^{\infty}\frac{j_{\nu, n}x^{-(\nu+1)}J_{\nu+1}\left(xj_{\nu, n}/b\right) J_{\nu}\left(yj_{\nu, n}/b\right)}{b^3 J_{\nu+1}^2(j_{\nu, n})}\exp\left(-\frac{j_{\nu, n}^2}{2b^2}t\right)
\end{align*}
and
\begin{align*}
&\left(m(x)\frac{\partial}{\partial x}\left(\frac{q^{(b)}_1(0, y, t, x)}{m(x)}\right)\right)\bigg|_{x=b}
=-q^{(b)}_2(t, y),\quad
\left(m(x)\frac{\partial}{\partial x}\left(\frac{q^{(b)}_1(0, y, t, x)}{m(x)}\right)\right)\bigg|_{x=0}
=0
\end{align*}
hold. Thus, we have 
\[-\frac{1}{2}\left[m(x)\frac{\partial}{\partial x}\left(\frac{q^{(b)}_1(0, y, t, x)}{m(x)}\right)\right]_{x=0}^{x=b}
=\frac{q^{(b)}_2(t, y)}{2}, \]
and \eqref{relation_q2_fht} is proved. 
By \eqref{relation_q2_fht}, Theorem~\ref{Thm_max_dist_of_BES_bridge}, and Corollary~\ref{Cor_Diff_max_dist_of_BES_bridge}, we easily obatin \eqref{relation_qbb_fht}.
\end{proof}

\begin{remark}
It is well-known that
\begin{align*}
&P\left(\tau_{a, b}\leq t\right)
=1-2\left(\frac{b}{a}\right)^{\nu}
\sum_{n=1}^{\infty}
\frac{J_{\nu}\left(aj_{\nu, n}/b\right)}{j_{\nu, n}J_{\nu+1}(j_{\nu, n})}\exp\left(-\frac{j_{\nu, n}^2}{2b^2}t\right),\\
&P\left(\tau_{0, b}\leq t\right)
=1-\frac{1}{2^{\nu-1}\Gamma(\nu+1)}
\sum_{n=1}^{\infty}\frac{j_{\nu, n}^{\nu-1}}{J_{\nu+1}(j_{\nu, n})}\exp\left(-\frac{j_{\nu, n}^2}{2b^2}t\right)
\end{align*}
hold for $0<a<b$ and $t>0$(\cite{bib_JT_Kent}). 
By differentiating these identities, we obtain 
\begin{align}
&P\left(\tau_{a, b}\in dt\right)
=\left(\frac{b}{a}\right)^{\nu}
\sum_{n=1}^{\infty}\frac{j_{\nu, n}J_{\nu}\left(aj_{\nu, n}/b\right)}{b^2 J_{\nu+1}(j_{\nu, n})}
\exp\left(-\frac{j_{\nu, n}^2}{2b^2}t\right)dt,
\label{BES_starting_a_fht_density}\\
&P\left(\tau_{0, b}\in dt\right)
=\frac{1}{2^{\nu}\Gamma(\nu+1)}
\sum_{n=1}^{\infty}\frac{j_{\nu, n}^{\nu+1}}{b^2 J_{\nu+1}(j_{\nu, n})}
\exp\left(-\frac{j_{\nu, n}^2}{2b^2}t\right)dt.
\label{BES_fht_density}
\end{align}
By using \eqref{BES_starting_a_fht_density} and \eqref{BES_fht_density}, we can also prove Lemma~\ref{Lem_relation_fht_q}. 
\end{remark}

\begin{theorem}
\label{Thm_moving_density_by_fht}
Let $0\leq a<b$. For $0<s<t<1$ and $x, y\in (0, b)$, we have 
\begin{align}
&P\left(R^a(t)\in dy~|~\tau_{a, b}=1\right)
=\frac{q^{(b)}_1(0, a, t, y)
q^{(b)}_2(1-t, y)}
{q^{(b)}_2(1, a)}dy,
\label{moving_dens_fht}\\
&P\left(R^a(t)\in dy~|~R^a(s)=x, \tau_{a, b}=1\right)
=\frac{q^{(b)}_1(s, x, t, y)
q^{(b)}_2(1-t, y)}
{q^{(b)}_2(1-s, x)}dy.
\label{moving_trans_dens_fht}
\end{align}
\end{theorem}
\begin{proof}
Using the Markov property of $R^a$, for $0< t< u$, it holds that 
\begin{align*}
P\left(R^a(t)\in dy, \tau_{a, b}> u\right)
&=P\left(R^a(t)\in dy, M_{u}(R^a)< b\right)\\
&=P\left(R^a(t)\in dy, M_t(R^a)< b\right)
P\left(M_{u-t}(R^y)< b\right)\\
&=P\left(R^a(t)\in dy, M_t(R^a)< b\right)
P\left(\tau_{y, b}> u-t\right).
\end{align*}
Since the density of $\tau_{a, b}$ is a derivative of the distribution function, we obtain 
\begin{align}\label{jointdens_BES_fht}
P\left(R^a(t)\in dy, \tau_{a, b}\in du\right)
&=-\frac{d}{du}P\left(R^a(t)\in dy, \tau_{a, b}> u\right)\notag \\
&=P\left(R^a(t)\in dy, M_t(R^a)< b\right)
P\left(\tau_{y, b}\in du-t\right).
\end{align}
We can calculate the first term of the right-hand side of \eqref{jointdens_BES_fht} as 
\begin{align}
P\left(R^a(t) \in dy, M_t(R^a) < b\right)
&=P\left(R^a(t) \in dy\right)\frac{P\left(R^a(t) \in dy, M_t(R^a) < b\right)}{P\left(R^a(t) \in dy\right)}
\nonumber \\
&=P\left(R^a(t) \in dy\right)P\left(M_{[0, t]}(r_{[0, t]}^{a\to y})< b\right)
\nonumber \\
&=q^{(b)}_1(0, a, t, y)dy.
\label{relation_q1_fht}
\end{align}
Therefore, by \eqref{jointdens_BES_fht}, \eqref{relation_q1_fht}, \eqref{relation_q2_fht}, \eqref{relation_qbb_fht}, and L'H\^{o}pital's rule, 
we can prove \eqref{moving_dens_fht} as follows: 
\begin{align*}
P\left(R^a(t)\in dy~|~\tau_{a, b}=1\right)
&=P\left(R^a(t)\in dy, M_t(R^a)< b\right)\frac{
P\left(\tau_{y, b}\in du-t\right)}
{P\left(\tau_{a, b}\in du\right)}\Big|_{u=1}\\
&=\frac{q^{(b)}_1(0, a, t, y)q^{(b)}_2(1-t, y)}{q^{(b)}_2(1, a)}dy.
\end{align*}
Next, we prove \eqref{moving_trans_dens_fht}. 
Using the Markov property of $R^a$, for $0< s<t< u$, it holds that 
\begin{align*}
&P\left(R^a(t)\in dy, R^a(s)\in dx, \tau_{a, b}> u\right)\\
&\quad=P\left(R^a(t)\in dy, R^a(s)\in dx, M_{u}(R^a)< b\right)\\
&\quad=P\left(R^a(s)\in dx, M_{s}(R^a)< b\right)
P\left(R^x(t-s)\in dy, M_{t-s}(R^x)< b\right)
P\left(M_{u-(t-s)}(R^y)< b\right)\\
&\quad=P\left(R^a(s)\in dx, M_{s}(R^a)< b\right)
P\left(R^x(t-s)\in dy, M_{t-s}(R^x)< b\right)
P\left(\tau_{y, b}> u-(t-s)\right).
\end{align*}
Thus, it follows that 
\begin{align*}
&P\left(R^a(t)\in dy, R^a(s)\in dx, \tau_{a, b}\in du\right)\\
&\quad=-\frac{d}{du}P\left(R^a(t)\in dy, R^a(s)\in dx, \tau_{a, b}> u\right)\\
&\quad=P\left(R^a(s)\in dx, M_{s}(R^a)< b\right)
P\left(R^x(t-s)\in dy, M_{t-s}(R^x)< b\right)
P\left(\tau_{y, b}\in du-(t-s)\right).
\end{align*}
On the other hand, by \eqref{jointdens_BES_fht}, we obtain 
\[
P\left(R^a(s)\in dx, \tau_{a, b}\in du\right)
=P\left(R^a(s)\in dx, M_{s}(R^a)< b\right)
P\left(\tau_{x, b}\in du-s\right). 
\]
Combining this equality, \eqref{relation_q1_fht}, \eqref{relation_q2_fht}, and L'H\^{o}pital's rule, we can prove \eqref{moving_trans_dens_fht} as follows:
\begin{align*}
&P\left(R^a(t)\in dy~|~R^a(s)=x, \tau_{a, b}=1\right)\\
&\quad=\frac{P\left(R^a(s)\in dx, M_{s}(R^a)< b\right)
P\left(R^x(t-s)\in dy, M_{t-s}(R^x)< b\right)
P\left(\tau_{y, b}\in du-(t-s)\right)}
{P\left(R^a(s)\in dx, M_{s}(R^a)< b\right)
P\left(\tau_{x, b}\in du-s\right)}
\Big|_{u=1}\\
&\quad=\frac{P\left(R^x(t-s)\in dy, M_{t-s}(R^x)< b\right)
P\left(\tau_{y, b}\in du-(t-s)\right)}
{P\left(\tau_{x, b}\in du-s\right)}
\Big|_{u=1}\\
&\quad=\frac{q^{(b)}_1(s, x, t, y)
q^{(b)}_2(1-t, y)}
{q^{(b)}_2(1-s, x)}dy.
\end{align*}
\end{proof}

According to Theorem~\ref{Thm_MainResult2}, the right sides of \eqref{moving_dens_fht} and \eqref{moving_trans_dens_fht} are the transition densities of $H^{a\to b}$. 
Therefore, the proof of Proposition~\ref{Prop_MainResult1} is completed.

\section{Proof of Proposition~\ref{Prop_Holder_conti_Moving}}
\label{Section_Holder}
The proof is similar to that in Chapter~2, 
Theorem 2.8 in \cite{bib_KarazasShreve}.
We fix $\gamma \in \left( 0, \frac{1}{2} \right)$. 
Then, we can find $m_0 \in \mathbb{N}$ 
so that 
$\gamma < \frac{m_0 - 3\nu-6}{2m_0}$ 
holds. 
For this $m_0$, combining 
Theorem~\ref{Thm_MainResult2}, 
Skorohod's theorem, Fatou's lemma, and Lemma~\ref{Lem_MomentEq}, 
we can take a positive number $C_{m_0, \nu, a, b}$ that satisfies
\begin{align*}
&E\left[ \left| H^{a \to b}(r) -a\right|^{2m_0} \right] 
\leq \frac{C_{m_0, \nu, a, b}}{r^{\nu+1-m_0}(1-r)^{\nu+\frac{5}{2}}},\\
&E\left[ \left| H^{a \to b}(1-r) - b \right|^{2m_0} \right] 
\leq \frac{C_{m_0, \nu, a, b}}{r^{\nu+2-m_0}(1-r)^{\nu+\frac{3}{2}}},\\
& E\left[ \left| H^{a \to b}(t) - H^{a \to b}(s) \right|^{2m_0} \right] 
\leq \frac{C_{m_0, \nu, a, b}}{(t-s)^{\nu+1-m_0}s^{\nu+\frac{3}{2}}(1-t)^{\nu+\frac{5}{2}}}
\end{align*}
for $t, s, r \in (0,1)$ with $s<t$. Now, for $n \in \mathbb{N}$, we define
\begin{align*}
&F_n = \left\{ 
\max_{1 \leq k \leq 2^n} 
\left| H^{a \to b} \left(\frac{k-1}{2^n}\right) 
- H^{a \to b} \left(\frac{k}{2^n}\right) \right| 
\geq 2^{-n \gamma} 
\right\}, \\
&a(n, k) = P\left( \left| H^{a \to b} \left(\frac{k-1}{2^n}\right) 
- H^{a \to b} \left(\frac{k}{2^n} \right) \right| \geq 2^{-n \gamma} \right),
\qquad 1 \leq k \leq 2^n.
\end{align*}
Then, Chebyshev's inequality yields
\begin{align*}
a(n, 1) 
&\leq 2^{2nm_0\gamma} 
E\left[ \left| H^{a \to b}(1/2^n) -a\right|^{2m_0} \right] 
\leq C_{m_0, \nu, a, b}2^{-n(m_0-2m_0\gamma-2\nu-\frac{7}{2})},\\
a(n, 2^n) 
&\leq 2^{2nm_0\gamma} 
E\left[ \left| H^{a \to b}(1-1/2^n) - b \right|^{2m_0} \right] 
\leq C_{m_0, \nu, a, b}2^{-n(m_0-2m_0\gamma-2\nu-\frac{7}{2})},
\end{align*}
and, for $2 \leq k \leq 2^n-1$,
\begin{align*}
a(n, k) 
&\leq 2^{2nm_0\gamma} 
E\left[ \left| H^{a \to b}((k-1)/2^n) - H^{a \to b}(k/2^n) \right|^{2m_0} \right] 
\leq C_{m_0, \nu, a, b} 2^{-n(m_0-2m_0\gamma-3\nu-5)}.
\end{align*}
 
Therefore, $P(F_n) \leq C_{m_0, \nu, a, b} \times 2^{-n(m_0 -2m_0\gamma-3\nu-6)}$, 
and because $m_0 -2m_0\gamma -3\nu-6> 0$, 
we have $P\left( \liminf_{n \to \infty} F_n^c \right) = 1$ 
by the first Borel--Cantelli lemma. 
If $\omega \in \liminf_{n \to \infty} F_n^c$, 
then there exists $n^*(\omega) \in \mathbb{N}$ 
such that $\omega \in \bigcap_{n \geq n^*(\omega)} F_n^c$. 
For $n \geq n^*(\omega)$, we can deduce that 
\begin{align*}
\left| H^{a \to b}(t)-H^{a \to b}(s) \right| 
\leq 2 \sum_{j=n+1}^{\infty} 2^{-\gamma j} 
= \frac{2}{1-2^{-\gamma}} 2^{-(n+1)\gamma},\qquad 0 < t-s < 2^{-n}.
\end{align*}
Now, let $t, s \in [0, 1]$ satisfy $0 < t-s < 2^{-n^*(\omega)}$ 
and choose $n \geq n^*(\omega)$ so that $2^{-(n+1)} \leq t-s < 2^{-n}$. 
Then, the above inequality yields
\begin{align*}
\left| H^{a \to b}(t)-H^{a \to b}(s) \right| \leq 
\frac{2}{1-2^{-\gamma}} \left| t-s \right|^{\gamma}.
\end{align*}
 Hence, $H^{a \to b}$ is locally H\"{o}lder-continuous 
with exponent $\gamma$ for $\omega \in \liminf_{n \to \infty} F_n^c$.
\qed

\section{The space-time reversal property of the BES($3$) house-moving and numerical examples}\label{section_numerical}
In this section, 
we show that the BES($3$) house-moving has the space-time reversal property. 

Although the following proposition is showed by \cite[Proposition~4.8]{bib_RevuzYor}, 
we prove it based on our setting for completeness.

\begin{prop}
Let $\delta=3\;(\nu=\frac{1}{2})$ and $b>0$. For $0<t<1$ and $y\in (0, b)$, we have 
\begin{align*}
P\left(H^{0\to b}(t)\in dy\right)
=P\left(H^{0\to b}(1-t)\in b-dy\right).
\end{align*}
\end{prop}
\begin{proof}
The Fourier expansion of the heat kernel shows that 
the following equality
\begin{align}
&\sum_{k=-\infty}^\infty (n_{t}(y-x+2k)-n_t(y+x+2k))
=2\sum_{n=1}^\infty\sin{xn\pi}\sin{yn\pi}\exp\left(-\frac{(n\pi)^2}{2}t\right)
\label{Eq_heatkernel}
\end{align}
holds for $x, y\in \mathbb{R}$ and $t>0$.
So, we obtain 
\begin{align}
2\sum_{k=-\infty}^\infty \frac{y+2k}{t}n_t(y+2k)
&=\lim_{x\downarrow0}
\sum_{k=-\infty}^\infty \frac{1}{x}(n_{t}(y-x+2k)-n_t(y+x+2k))
\notag \\
&=\lim_{x\downarrow0}2\sum_{n=1}^\infty\frac{\sin{xn\pi}}{x}\sin{yn\pi}\exp\left(-\frac{(n\pi)^2}{2}t\right)
\notag \\
&=2\sum_{n=1}^\infty n\pi\sin{yn\pi}\exp\left(-\frac{(n\pi)^2}{2}t\right), 
\quad y\in \mathbb{R}, \ t>0.
\label{Eq_heatkernel2}
\end{align}
Also, in \cite{bib_Watson}, we have 
\begin{align}
&J_{\frac{1}{2}}(x)=\sqrt{\frac{2}{\pi x}}\sin{x},\quad J_{\frac{3}{2}}(x)=\sqrt{\frac{2}{\pi x}}\left(\frac{\sin{x}}{x}-\cos{x}\right),
\quad j_{\frac{1}{2}, n}=n\pi \quad (n\in\mathbb{N}).
\label{J_1/2_3/2}
\end{align}
Using \eqref{Eq_heatkernel}, \eqref{Eq_heatkernel2}, \eqref{J_1/2_3/2},
and Theorem~\ref{Thm_max_dist_of_BES_bridge}, for $0< s < t < 1$, $0<x, y< \eta$, 
we can obtain the expressions for $q^{(\eta)}_1(s, x, t, y)$ and $q^{(\eta)}_1(0, 0, t, y)$ in the case of $\nu=\frac{1}{2}$: 
\begin{align}
q^{(\eta)}_1(s, x, t, y)
&=\frac{y}{x}\sum_{k=-\infty}^\infty(n_{t-s}(y-x+2k\eta)-n_{t-s}(y+x+2k\eta)), 
\label{q1_12}\\
q^{(\eta)}_1(0, 0, t, y)
&=y\sum_{k=-\infty}^\infty2\frac{y+2k\eta}{t}n_t(y+2k\eta).
\label{q1_00}
\end{align}
On the other hand, 
according to \eqref{q1_12}, we obtain 
\begin{align}
q^{(\eta+\varepsilon)}_1(t, y, 1, \eta)
&=\frac{\eta}{y}\sum_{k=-\infty}^\infty(n_{1-t}(\eta-y+2k(\eta+\varepsilon))-n_{1-t}(y+(2k+1)\eta+2k\varepsilon))\notag \\
&=\frac{\eta}{y}\sum_{k=-\infty}^\infty(n_{1-t}(\eta-y+2k(\eta+\varepsilon))-n_{1-t}(\eta-y-2(k+1)\eta-2k\varepsilon))\notag \\
&=\frac{\eta}{y}\sum_{k=-\infty}^\infty(n_{1-t}(\eta-y+2k(\eta+\varepsilon))-n_{1-t}(\eta-y+2k\eta+2(k+1)\varepsilon))
\label{q1_999}
\end{align}
for $\varepsilon>0$. 
Using \eqref{q1_00} and \eqref{q1_999}, we get 
\begin{align*}
&q^{(\eta)}_2(1-t, y)\\
&\quad =
\lim_{\varepsilon\downarrow0}\frac{\partial}{\partial \varepsilon}q^{(\eta+\varepsilon)}_1(t, y, 1, \eta)\\
&\quad =
\frac{\eta}{y}\sum_{k=-\infty}^\infty\left(-2k\frac{\eta-y+2k\eta}{1-t}n_{1-t}(\eta-y+2k\eta)+2(k+1)\frac{\eta-y+2k\eta}{1-t}n_{1-t}(\eta-y+2k\eta)\right)\\
&\quad =
\frac{\eta}{y}\sum_{k=-\infty}^\infty2\frac{\eta-y+2k\eta}{1-t}n_{1-t}(\eta-y+2k\eta)\\
&\quad =
\frac{\eta}{y(\eta-y)}q^{(\eta)}_1(0, 0, 1-t, \eta-y).
\end{align*}
Thus, it holds that 
\begin{align*}
P\left(H^{0\to b}(t)\in dy\right)
&=\frac{q^{(b)}_1(0, 0, t, y)q^{(b)}_2(1-t, y)}{q^{(b)}_2(1, 0)}dy\\
&=\frac{b}{y(b-y)}\cdot \frac{q^{(b)}_1(0, 0, t, y)q^{(b)}_1(0, 0, 1-t, b-y)}{q^{(b)}_2(1, 0)}dy\\
&=P\left(H^{0\to b}(1-t)\in b-dy\right),
\end{align*}
and the proof is completed.
\end{proof}

\section{Future work}\label{section_Future_work}

We are interested in finding the stochastic differential equations for the BES($\delta$) house-moving.
We are currently investigating this problem by using Theorem~\ref{Thm_abs_conti}.

In addition, 
let $R=\{R(t)\}_{t\geq 0}$ be a regular one-dimensional diffusion on $[0, \infty)$. 
For an $R$-bridge 
$r^{0\to b}=\{r^{0\to b}(s)\}_{s\in [0, 1]}$ $(b>0)$ from $0$ to $b$ on $[0,1]$, 
we are also interested in the weak convergence of 
$r^{0 \to b}|_{K^-(b+\eta)}$ as $\eta \downarrow 0$.

\appendix
\section{Appendix}\label{section_appendix}

\subsection{Bessel functions}\label{subsection_BES_function}
Let $J_{\alpha}(z)$ and $I_{\alpha}(z)$ denote the Bessel function 
and modified Bessel function of the first kind with index $\alpha\in\mathbb{R}$, respectively. 
They are defined as 
\begin{align*}
J_{\alpha}(z)&=\left(\frac{1}{2}z\right)^{\alpha}
\sum_{k\in \mathbb{Z}_+}
\frac{\left(-\frac{1}{4}z^2\right)^k}{k!\Gamma(\alpha+k+1)},\\
I_{\alpha}(z)&=\left(\frac{1}{2}z\right)^{\alpha}
\sum_{k\in \mathbb{Z}_+}
\frac{\left(\frac{1}{4}z^2\right)^k}{k! \Gamma(\alpha+k+1)}
\end{align*}
for $z\in\mathbb{C}\setminus \mathbb{R}_-$. 
In addition, for $z\in \mathbb{C}\setminus\mathbb{R}_-$, we define 
\[
K_{\alpha}(z):=\frac{\pi\left(I_{-\alpha}(z)-I_{\alpha}(z)\right)}{2\sin(\alpha\pi)},
\]
when $\alpha\in\mathbb{R}\setminus \mathbb{Z}$, and
\[
K_{\alpha}(z):=\lim_{\beta\to \alpha}K_{\beta}(z)
\]
when $\alpha\in\mathbb{Z}$. 
$K_{\alpha}(z)$ is called the modified Bessel function with index $\alpha$ of the second kind. 
Moreover, the values of $z^{-\alpha}J_{\alpha}(z)$ and $z^{-\alpha}I_{\alpha}(z)$ at zero are written as 
\begin{align*}
&z^{-\alpha}J_{\alpha}(z)\vert_{z=0}
=\frac{1}{2^{\alpha}\Gamma(\alpha+1)}, \quad
z^{-\alpha}I_{\alpha}(z)\vert_{z=0}
=\frac{1}{2^{\alpha}\Gamma(\alpha+1)}.
\end{align*}

We obtain the following derivatives:
\begin{align}
\frac{d}{dz}\left(z^{\alpha}J_{\alpha}(z)\right)=z^{\alpha}J_{\alpha-1}(z), \quad
&\frac{d}{dz}\left(z^{-\alpha}J_{\alpha}(z)\right)=-z^{-\alpha}J_{\alpha+1}(z),
\quad z\in \mathbb{C}\setminus\mathbb{R}_-, 
\label{diffeq_z_alpha_J_alpha_z}\\
\frac{d}{dz}\left(z^{\alpha}I_{\alpha}(z)\right)=z^{\alpha}I_{\alpha-1}(z), \quad 
&\frac{d}{dz}\left(z^{-\alpha}I_{\alpha}(z)\right)=z^{-\alpha}I_{\alpha+1}(z), 
\quad z\in \mathbb{C}\setminus\mathbb{R}_-.
\label{diffeq_z_alpha_I_alpha_z}
\end{align}
Moreover, using \eqref{diffeq_z_alpha_J_alpha_z} and \eqref{diffeq_z_alpha_I_alpha_z}, we have 
\begin{align*}
&\frac{d}{dz}J_{\alpha}(z)=J_{\alpha-1}(z)-\frac{\alpha}{z}J_{\alpha}(z)
=-J_{\alpha+1}(z)+\frac{\alpha}{z}J_{\alpha}(z),\\
&\frac{d}{dz}I_{\alpha}(z)=I_{\alpha-1}(z)-\frac{\alpha}{z}I_{\alpha}(z)
=I_{\alpha+1}(z)+\frac{\alpha}{z}I_{\alpha}(z).
\end{align*}

In the rest of this section, we assume that $\alpha>-1$. 
According to \cite[$(2.2)$, and $(2.8)$]{bib_Serafin}, 
there exists $C_{\alpha}>0$ such that 
\begin{align}
&z^{-\alpha}|J_{\alpha}(z)|\leq 
C_{\alpha}\frac{1}{(1+z)^{\alpha+\frac{1}{2}}},\quad z\geq 0,
\label{esti_J}\\
&z^{-\alpha} I_{\alpha}(z)\leq 
C_{\alpha}\frac{1}{(1+z)^{\alpha+\frac{1}{2}}}e^z,\quad z\geq 0.
\label{esti_I}
\end{align}
The sequence of positive zeros of the Bessel function $J_{\alpha}$ is denoted by $\{j_{\alpha, n}\}_{n=1}^{\infty}$ ($j_{\alpha, 1}<j_{\alpha, 2}<\cdots$). 
According to \cite{bib_Watson}, we have
\begin{align*}
J_{\alpha+1}(j_{\alpha, n})\neq 0 
\quad \mbox{and}\quad j_{\alpha, n}<j_{\alpha+1, n}<j_{\alpha, n+1}
\qquad (n=1, 2, \ldots).
\end{align*}
In addition, from \cite{bib_Watson} we find the following asymptotics as $n\to\infty$: 
\begin{align}
j_{\alpha, n}\sim n\pi,\quad
J_{\alpha+1}(j_{\alpha, n})\sim (-1)^{n-1}\sqrt{\frac{2}{\pi j_{\alpha, n}}}\sim (-1)^{n-1}\frac{1}{\pi}\sqrt{\frac{2}{n}}.
\label{Esti_zeros}
\end{align}

\subsection{General results on continuous processes}

In this subsection, we introduce some general results used in this paper. 
The proofs of them are found in \cite{bib_3DimMoving}.

\begin{theorem}[{\cite[Chapter 2, Theorem 4.15]{bib_KarazasShreve}}]
\label{Ap_Thm_KS_4.15}
Let $\{X_n \}_{n=1}^{\infty}$ be the family of 
$C([0,1], \mathbb{R}^d)$-valued random variables. 
If the family $\{X_n \}_{n=1}^{\infty}$ is tight 
and the finite-dimensional distribution of $X_n$ 
converges to that of some $X$, then $X_n \overset{\mathcal{D}}{\to} X$ holds.
\end{theorem}

\begin{lem}[{\cite[Appendix]{bib_3DimMoving}}]\label{Ap_Lem_Markov_TransDensityConv_FinDimDistConv}
Let $a, b\in \mathbb{R}^d$, 
and let $X_n$ and $X$ are $\mathbb{R}^d$-valued Markovian bridges from $a$ to $b$ on $[0,1]$
for $n\in \mathbb{N}$. 
Let $X_n$ and $X$ have the respective transition densities
\begin{align*}
P\left( X_n(t) \in dy \right) = q_n(t,y)dy,
&\qquad P\left( X_n(t) \in dy~|~X_n(s)=x \right) = q_n(s,x,t,y)dy,\\
P\left( X(t) \in dy \right) = q(t,y)dy, 
&\qquad P\left( X(t) \in dy~|~X(s)=x \right) = q(s,x,t,y)dy
\end{align*}
for $0 < s < t < 1, x,y \in \mathbb{R}^d$, and $n\in \mathbb{N}$. 
If 
\begin{align*}
\lim_{n \to \infty} q_n(t,y) = q(t,y)&,\qquad a.e.  \ y \in \mathbb{R}^d,\\
\lim_{n \to \infty} q_n(s,x,t,y) = q(s,x,t,y)&,
\qquad a.e. \ (x,y)\in \mathbb{R}^d \times \mathbb{R}^d,
\end{align*}
for $0< s<t< 1$, 
then the finite-dimensional distribution of $X_n$ 
converges to that of $X$ as $n \to \infty$.
\end{lem}

\begin{theorem}[{\cite[Appendix]{bib_3DimMoving}}]\label{Ap_Thm_suf_cond_restricted_tightness}
For $\varepsilon \in \mathcal{E}$, $X^{(\varepsilon )}$ is a 
$(C([0, 1], \mathbb{R}^d), \mathcal{B}(C([0, 1], \mathbb{R}^d)))$-valued 
random variable defined on 
$(\Omega^{(\varepsilon)}, \mathcal{F}^{(\varepsilon)}, P^{(\varepsilon)})$. 
Assume that $\{ X^{(\varepsilon )}(0) \}_{\varepsilon \in \mathcal{E}}$ 
is uniformly integrable and that the following conditions hold: 
\begin{enumerate}[{\rm (1)}]
\item 
For each $u \in \left(0,\frac{1}{2} \right)$, 
$\{ \pi_{[u,1-u]} \circ X^{(\varepsilon )} \}_{\varepsilon \in \mathcal{E}}$ 
is tight.
\item 
For each $\xi >0$, it holds that
\begin{align*}
&\lim_{u \downarrow 0} \sup_{\varepsilon \in \mathcal{E}} 
P^{(\varepsilon)} \left( \sup_{0 \leq t \leq u} 
|X^{(\varepsilon )}(t) - X^{(\varepsilon )}(0)| > \xi \right) = 0, \\
&\lim_{u \downarrow 0} \sup_{\varepsilon \in \mathcal{E}} 
P^{(\varepsilon)} \left( \sup_{1-u \leq t \leq 1} 
|X^{(\varepsilon )}(t) - X^{(\varepsilon )}(1)| > \xi \right) = 0.
\end{align*}
\end{enumerate}
Then, the family $\{ X^{(\varepsilon )}\}_{\varepsilon \in \mathcal{E}}$ is tight.
\end{theorem}

\begin{lem}{(Chapter 2, Problem 4.11 in \cite{bib_KarazasShreve})}
\label{Ap_Lem_Moment_estimate_type_suf_cond_for_tightness}
For $\varepsilon \in \mathcal{E}$, $X^{(\varepsilon )}$ is a 
$(C([0, 1], \mathbb{R}^d), \mathcal{B}(C([0, 1], \mathbb{R}^d)))$-valued 
random variable defined on 
$(\Omega^{(\varepsilon)}, \mathcal{F}^{(\varepsilon)}, P^{(\varepsilon)})$. 
Assume that 
$\{X^{(\varepsilon )} \}_{\varepsilon \in \mathcal{E}}$ 
satisfies the following conditions:
\begin{enumerate}[{\rm (1)}]
\item There exists some $\nu > 0$ that satisfies
\begin{align*}
\sup_{\varepsilon \in \mathcal{E}} 
E^{(\varepsilon)} \left[ \left| X^{(\varepsilon )}(0) \right|^{\nu} \right] < \infty.
\end{align*}
\item There exist $\alpha, \beta, C>0$ that satisfy
\begin{align*}
\sup_{\varepsilon \in \mathcal{E}} 
E^{(\varepsilon)} 
\left[ \left| X^{(\varepsilon )}(t) - X^{(\varepsilon )}(s) \right|^{\alpha} \right] 
\leq C \left| t-s \right|^{1+\beta}, 
\qquad t,s \in [0,1].
\end{align*}
\end{enumerate}
Then $\{ X^{(\varepsilon )}\}_{\varepsilon \in \mathcal{E}} $ is tight.
\end{lem}

\begin{lem}[{\cite[Appendix]{bib_3DimMoving}}]\label{Ap_Lem_unif_conti_at_t=0}
Let $\gamma>0$. 
For $\varepsilon \in \mathcal{E}$, $X^{(\varepsilon )}$ is a 
$(C([0, 1], \mathbb{R}^d), \mathcal{B}(C([0, 1], \mathbb{R}^d)))$-valued 
random variable defined on 
$(\Omega^{(\varepsilon)}, \mathcal{F}^{(\varepsilon)}, P^{(\varepsilon)})$. 
Assume that
\begin{align*}
F_l^{\varepsilon} := 
\left\{ \max_{1 \leq k \leq 2^{l-1}} 
\left| X^{(\varepsilon )} \left(\frac{k-1}{2^l} \right) 
- X^{(\varepsilon )} \left( \frac{k}{2^l} \right) \right| 
\geq 2^{-l\gamma} \right\} 
\in \mathcal{F}^{(\varepsilon)},
\qquad \varepsilon \in \mathcal{E}, \quad l=1,2,\ldots
\end{align*}
satisfy $\sum_{l=1}^{\infty} 
\sup_{\varepsilon \in \mathcal{E}} 
P^{(\varepsilon )}(F_l^{\varepsilon}) < \infty$, then we have
\begin{align*}
\lim_{u \downarrow 0} \sup_{\varepsilon \in \mathcal{E}} 
P^{(\varepsilon )} \left( 
\sup_{0 \leq t \leq u} 
\left| X^{(\varepsilon )}(t) - X^{(\varepsilon )}(0) \right| 
>\xi \right) = 0,
\qquad \xi>0.
\end{align*}
\end{lem}

\begin{lem}[{\cite[Appendix]{bib_3DimMoving}}]\label{Ap_Lem_unif_conti_at_t=1}
Under the same assumption of 
Lemma~\ref{Ap_Lem_unif_conti_at_t=0}, if 
\begin{align*}
\widetilde{F}_l^{\varepsilon} = 
\left\{ \max_{2^{l-1} \leq k \leq 2^l} 
\left| X^{(\varepsilon )} \left(\frac{k-1}{2^l} \right) 
- X^{(\varepsilon )} \left( \frac{k}{2^l} \right) \right| 
\geq 2^{-l\gamma} \right\}
\in \mathcal{F}^{(\varepsilon)},
\qquad \varepsilon \in \mathcal{E}, \quad l=1,2,\ldots
\end{align*}
satisfy $\sum_{l=1}^{\infty} \sup_{\varepsilon \in \mathcal{E}} 
P^{(\varepsilon )}(\widetilde{F}_l^{\varepsilon}) < \infty$, then we have
\begin{align*}
\lim_{u \downarrow 0} \sup_{\varepsilon \in \mathcal{E}} 
P^{(\varepsilon )} \left( 
\sup_{0 \leq t \leq u} 
\left| X^{(\varepsilon )}(1-t) - X^{(\varepsilon )}(1) \right| >\xi \right) = 0,
\qquad \xi>0.
\end{align*}
\end{lem}

\begin{lem}[{\cite[Appendix]{bib_3DimMoving}}]\label{Ap_Lem_weakconv_product}
Let $S_1$ and $S_2$ be Polish spaces, 
and let $X_n$ and $Y_n$ be random variables 
defined on $(\Omega_n, \mathcal{F}_n, P_n)$ 
that take values in $S_1$ and $S_2$, respectively. 
If $X_n$ and $Y_n$ are independent and 
$P_n \circ X_n^{-1}$ and $P_n \circ Y_n^{-1}$ converge to 
probability measures $Q$ on $S_1$ and $R$ on $S_2$, respectively, 
then $P_n \circ (X_n, Y_n)^{-1}$ 
converges to the product measure $Q \times R$.
\end{lem}

\section*{\rm Acknowledgments}

The authors thank Prof.\ Kumiko Hattori (Tokyo Metropolitan University) 
for helpful comments and discussions on the subject of this paper. 
This research was supported by JSPS KAKENHI Grant Number JP22K01556, 
a grant-in-aid from the Zengin Foundation for Studies on Economics and Finance. 
This research was also supported by JST, the establishment of university
fellowships towards the creation of science technology innovation, Grant Number JPMJFS2139, 
Grant-in-Aid for JSPS Fellows Grant Number JP23KJ1801. 
The authors also thank the reviewer for various comments and constructive suggestions to improve the quality of the paper.

\newpage
\begin{flushleft}
\mbox{  }\\
\hspace{95mm} Kensuke Ishitani\\
\hspace{95mm} Department of Mathematical Sciences\\
\hspace{95mm} Tokyo Metropolitan University\\
\hspace{95mm} Hachioji, Tokyo 192-0397\\
\hspace{95mm} Japan\\
\hspace{95mm} e-mail: k-ishitani@tmu.ac.jp\\
\mbox{  }\\
\hspace{95mm} Tokufuku Rin\\
\hspace{95mm} Kanpo System Solutions\\
\hspace{95mm} Shinagawa, Tokyo 141-0001\\
\hspace{95mm} Japan\\
\mbox{  }\\
\hspace{95mm} Shun Yanashima\\
\hspace{95mm} Department of Mathematical Sciences\\
\hspace{95mm} Tokyo Metropolitan University\\
\hspace{95mm} Hachioji, Tokyo 192-0397\\
\hspace{95mm} Japan\\
\hspace{95mm} e-mail: yanashima-shun@ed.tmu.ac.jp\\
\end{flushleft}

\end{document}